\documentclass[twoside,12pt]{article}
\usepackage{amsmath}
\usepackage{amsfonts}\usepackage{amsthm}
\usepackage{graphicx,pgf}
\usepackage{tikz}
\newtheorem{thm}{Theorem}[section]

\newtheorem{lemma}[thm]{Lemma}
\newtheorem{remark}[thm]{Remark}

\renewcommand {\theequation}{\thesection.\arabic{equation}}

\newcommand{\ds}{\displaystyle}
\newcommand{\beq}{\begin{equation}}
\newcommand{\eeq}{\end{equation}}

\usepackage{epsfig}
\usepackage{amssymb}
\usepackage{amsmath}
\usepackage{amsthm}
\usepackage{aliascnt}
\usepackage{indentfirst}
\usepackage{footnpag}
\usepackage{bm}
\usepackage{graphicx,pgf}
\usepackage{tikz}

\numberwithin{equation}{section}
\theoremstyle{definition}

\newtheorem{*theorem}[theorem]{$\bullet$ Theorem}

\newtheorem{*corollary}[theorem]{$\bullet$ Corollary}

\newtheorem*{definition*}{Definition}

\theoremstyle{remark}

\newtheorem{remark*}[rem]{$\star$ Remark}
\usepackage[numbers]{natbib}
\newcommand{\abs}[1]{\lvert#1\rvert}

\newcommand{\e}{\epsilon}

\begin{document}
\title{{\bf The $\bf \Gamma$-limit 
of traveling waves in the FitzHugh-Nagumo system}}

\author{Chao-Nien Chen \thanks{Department of Mathematics, National Tsing Hua
University, Hsinchu, Taiwan, ROC ({\tt chen@math.nthu.edu.tw)}.}
\and Y.S. Choi \thanks{Department of Mathematics, University of Connecticut,
                    Storrs, CT 06269-3009  ({\tt choi@math.uconn.edu)}.}
   \and Nicola Fusco \thanks{Dipartimento di Matematica e Applicazioni "R. Caccioppoli", Universita` degli Studi di Napoli "Federico II",
via Cintia, Monte S. Angelo, IT-80126 Napoli, Italy ({\tt n.fusco@unina.it)}.}} 
\date{}
\maketitle
\begin{quote}
\textbf{Abstract}: 
Patterns and waves are  basic and important 
{phenomena} that govern the dynamics of physical and biological systems. A common {theme} in investigating such systems 
 is to {identify} the intrinsic factors 
 responsible for such self-organization. The $\Gamma$-convergence is a well-known {technique applicable to} variational formulation in studying the concentration phenomena of stable patterns. A geometric variational functional  {associated} with the $\Gamma$-limit 
of standing waves of FitzHugh-Nagumo system has recently been {built}.
{This article studies the $\Gamma$-limit of traveling waves. To the best of our knowledge, this is the first attempt to 
expand the scope of applicability of $\Gamma$-convergence to cover non-stationary problems.}



\end{quote}

\section{Introduction} 
\setcounter{equation}{0}
\renewcommand{\theequation}{\thesection.\arabic{equation}}

Patterns and waves are  basic and important 
{phenomena} \cite{AFM,CELK,CS,LMN,NTYU,RW1,WW,Y2} that govern the dynamics of physical and biological systems. 
{These can be seen in}, for instance, morphological phases in block copolymers, 
skin pigmentation in cell development and semiconductor gas-discharge systems. 
In the investigation of such systems, a common {theme} 
 is to {identify} the intrinsic factors 
 {responsible for such self-organization}. For the reaction-diffusion systems, self-organized patterns have not only been  
   found in the neighborhoods of Turing's instability \cite{T}, but recent works \cite{CCH,CC1,CC2,CC3,CH,CH1,CKM,CT,DHK,RS,vCNT} 
    exhibited that some patterns and waves 
    possess 
    localized spatial 
     structures. 
In fact localized waves in reaction-diffusion systems 
  are commonly observed, referred to as dissipative solitons \cite{BLSP,F,L,NAY} in physical literature. \\

The FitzHugh-Nagumo model, which was originally derived as an excitable system for studying 
nerve impulse propagation, is now of great interest to the scientific community as breeding grounds for patterns, traveling waves, and other localized structures. It has been extensively studied as a paradigmatic activator-inhibitor system 
for patterns generated 
from homogeneous media destabilized by a spatial modulation. 
These patterns are robust in the sense that they are stable and exist for
a wide range of parameters. \\ 

The $\Gamma$-convergence \cite{Braides2,Le} is a well-known {technique applicable} to variational formulation in studying the concentration phenomena of stable patterns. 
When {a stationary} FitzHugh-Nagumo system
{is equipped} with an appropriate scaling on the parameters, we are led to studying a geometric variational problem \cite{CCR,CCHR} 
 with a $\Gamma$-limit 
   energy functional defined by 
\begin{equation}
  {\cal J}_D(\Omega)={\cal P}_D(\Omega)-\alpha\abs{\Omega}+\frac{\sigma}{2}\int_{\Omega}\mathcal{N}_D(\Omega)dx,
  \label{JD}
\end{equation}
where  
$\Omega$ is a measurable subset of the domain $D \subset {\mathbb R}^N$. 
Denoted by $|\Omega|$ its Lebesgue measure and ${\cal P}_D (\Omega)$ 
the perimeter of $\Omega$ in $D$.
In case $\Omega$ is of class $C^1$, ${\cal P}_D(\Omega)$ is 
the area of the part of the boundary of $\Omega$ that is inside $D$. 
{The integral term represents nonlocal influence with ${\mathcal N}_D(\Omega)$ satisfying the linear equation
\begin{equation} \label{ND_eqn}
- \Delta {\mathcal N}_D (\Omega)+ {\mathcal N}_D(\Omega)= \chi_\Omega\, ; \quad \partial_\nu {\mathcal N}_D(\Omega)=0.
\end{equation}
Here $\partial_\nu$ is the outward normal derivative.} \\

We will build a geometric variational functional associated with traveling wave {investigation} in later sections. To understand this derivation 
as opposed to that of the
stationary problem, we give a {brief} 
review on
the connection between  (\ref{JD}) and the FitzHugh-Nagumo
system. {Consider} 
\begin{eqnarray}
 u_t &=&   \e^2 \Delta u - u\Big  (u-\frac{1}{2} \Big ) (u-1) + \e \alpha
   - \e \sigma v, \label{fn-u} \\
 v_t &=&   \Delta v - v + u.  \label{fn-v}
\end{eqnarray}
With $\alpha>0$ and $\sigma >0$ being fixed, 
 a small $\e$ identifies a range where a singular limit will emerge.
 Recall that $u$ acts as an activator and $v$ is the inhibitor.
 Physically $\alpha$ measures the driving force towards a non-trivial state while $\sigma$ represents
the stablizing inhibition mechanism. Their competition leads to interesting dynamics and the emergence of 
 patterns. In dealing with stationary solutions of (\ref{fn-u})-(\ref{fn-v}) 
 both $u_t$ and $v_t$ {vanish}. Solving (\ref{fn-v}) for $v$ in terms of $u$ and denoted this solution by {$v = {\cal L}_D u$}, we see that 
(\ref{fn-u}) becomes
\begin{equation}
  -\e^2 \Delta u + u \Big ( u -\frac{1}{2} \Big ) (u-1) - \e \alpha
  + \e \sigma {\cal L}_D u =0 \ \ . 
  \label{diff-integro}
\end{equation}
The solutions of (\ref{diff-integro}) 
are the critical points of 
\begin{equation}
  {\cal I}_{D,\e} (u) = \int_D \Big ( \frac{\e^2}{2} |\nabla u|^2
     + \frac{u^2 (u-1)^2}{4}
     - \e \alpha u + \frac{\e \sigma }{2} u {\cal L}_D u \Big ) \, dx. 
  \label{Ie}
\end{equation}
When $D$ is bounded, 
$\e^{-1} {\cal I}_{D,\e}$ $\Gamma$-converges {in $L^1(D)$} to 
\begin{equation}
 \frac{\sqrt{2}}{12} {\cal P}_D(\Omega) - \alpha |\Omega| +
    \frac{\sigma}{2} \int_\Omega{\cal N}_D (\Omega) \, dx, 
  \label{gammalimit}
\end{equation}
a functional 
equivalent to (\ref{JD}). In case {$D = \mathbb{R}^N$,} a ball shaped stationary
set of {${\cal J}_{{\mathbb R}^N}$} is referred as a bubble or an entire solution. \\   

Front and pulse are localized waves, the latter 
 is manifest {as} a small spot. In the past, 
 $\Gamma$-convergence has been employed to establish many interesting results for stable patterns; however to the best of our knowledge, this tool has not been utilized to treat 
  traveling waves. We make attempt toward {this} goal, starting with the investigation of 
   planar traveling wave solutions of the following FitzHugh-Nagumo system: 
\begin{eqnarray} \label{main11}
\left\{  \begin{array}{rl}
\displaystyle  u_t &= \Delta u+\frac{1}{d} (f_\epsilon(u)-\epsilon \sigma v)  \; , \\ \\
\displaystyle v_t &= \Delta v+u-{\gamma} v \;,
\end{array} \right.
\end{eqnarray}
where $f_\epsilon(\xi) =-\xi(\xi-\beta_\epsilon)(\xi-1)$, $\beta_\epsilon= \frac{1}{2} - \frac{\alpha \epsilon}{\sqrt{2}}$ and $\alpha,\gamma,\sigma>0$. 
{As a remark, with $d=\epsilon^2$ and $\gamma=1$, it can be shown as $\epsilon \to 0$,
 its stationary problem leads to the same geometric variational functional (\ref{gammalimit}), 
if we replace $\alpha$ by $\alpha/6\sqrt{2}$ and work on a bound domain.}  \\

A planar traveling wave solution is of the from $(u(x-ct),v(x-ct))$; that is, the wave moves {with a speed} $c$ and keeps the same profile along the moving coordinates. {To treat traveling wave using  $\Gamma$-convergence},
the ansatz $(u(c(x-ct)),v(c(x-ct)))$, proposed in \cite{He}, {is} more appropriate 
and will be {adopted} in the paper.
We are thus led to deal with 
\begin{eqnarray}
dc^2 u_{xx} + dc^2 u_x  +{f_\epsilon(u)- \epsilon \sigma} v& = &0 \;, \label{dc1} \\
c^2 v_{xx} +c^2 v_x  -{\gamma} v +u &=&0  \;, \label{dc2}
\end{eqnarray}
with 
the value of $c$ 
  to be determined. Let us remark that (\ref{JD}) {results from setting $\gamma=1$.} \\ 

Let $L^2_e=\{ w: \int_{-\infty}^\infty e^x w^2 \,dx < \infty \}$.
In studying the $\Gamma$-convergence {in $L^2_e$ topology of the traveling wave functional associated} with \eqref{dc1}-\eqref{dc2}, 
the quantity $\sqrt {dc^2}$ {serves the role of a small} 
 parameter {$\epsilon$, the conventional notation being used}. 
{With 
$ \int_0^{1} f_\epsilon(\xi)  \; d\xi > 0$ for all small $\epsilon$, we seek 
traveling wave solutions with $c>0$.} 
 Given a prescribed velocity $c_0 > 0$, we seek traveling wave with speed 
 $c_0+o(1)$ for small $\epsilon$; 
  to be more precise, $d=\epsilon^2/c^2(\epsilon)$ with $c(\epsilon) \to c_0$ as $\epsilon \to 0$ 
  and $c_0$ indeed is 
  determined by other parameters. 
  {For the existence of traveling waves whose $L^2_e$-limit is a front},
 it is assumed that the following condition holds: \\

\noindent
{{($A 1$)}} $\quad \alpha > \frac{3 \sqrt{2} \sigma}{\gamma} > \alpha-1 > 0$.  \\
  
  Let $h_*=1 - \frac{(\alpha-1) \gamma}{3 \sqrt{2} \sigma}$,  
\begin{equation} \label{cinfty1}
c_f= \frac{2 h_* \sqrt{\gamma}}{\sqrt{1-h_*^2}}  \;.
\end{equation}
{A translation of a traveling wave solution remains a solution. For concreteness
we impose an additional constraint
 $\| u_\epsilon \|_{L^2_e}=1$ when looking for such waves.}

\begin{thm} \label{thm_front}
Assumed that {($A 1$)} is satisfied. 
If $c_f$ satisfies \eqref{cinfty1} for given $\sigma,\gamma$ and $\alpha$, 
there is an $\epsilon_0>0$
such that if $\epsilon \leq \epsilon_0$, {there are  $c_\epsilon$, $d=\epsilon^2/c^2_\epsilon$ for which} 
a traveling wave 
 solution $(c_\epsilon,u_\epsilon,v_\epsilon)$ of \eqref{dc1}-\eqref{dc2} 
  exists {with $\| u_\epsilon \|_{L^2_e}=1$}. 
   Moreover $c_\epsilon \to c_f$ and $u_\epsilon \to \chi_{(-\infty,0]}$ in $L^2_e$ as $\epsilon \to 0$.
\end{thm}


 In Theorem \ref{thm_front}, $c_f$ is unique determined when $\sigma,\gamma$ and $\alpha$ are given, and for small $\epsilon$ the {speeds}
 ${c_\epsilon}$ of traveling wave solutions {are} known 
{to leading order}. {Since the limit is unique, the convergence takes place along the whole sequence}. \\ 

{Next we investigate the question that the {$L^2_e$-limit} is a 
traveling pulse.} 
In this case, the parameters satisfy the following condition: \\ 

\noindent
 {($A 2$) $\quad 
 \frac{3 \sqrt{2} \sigma}{\gamma} > \alpha >1$}. \\ 

\begin{thm} \label{thm_epsilon_pulse}
Assumed that {($A 2$)} is satisfied. 
{There} is an $\epsilon_0>0$ such that if $\epsilon \leq \epsilon_0$, {there are $c_\epsilon$ and $d=\epsilon^2/c^2_\epsilon$
for which} 
a traveling wave 
 solution $(c_\epsilon,u_\epsilon,v_\epsilon)$ of \eqref{dc1}-\eqref{dc2} 
  exists {with $\| u_\epsilon \|_{L^2_e}=1$}. 
 Moreover 
 if $\epsilon \to 0$ then $c_\epsilon \to c_p$ and $u_\epsilon \to \chi_{[a,b]}$ in $L^2_e$, 
 where {$c_p$ and $b-a$ are}
  uniquely determined by the given parameters $\sigma,\gamma$ and $\alpha$.
\end{thm}

It was proved \cite{CCR,CCHR} that {when $N=1$ there always exists a single bubble;}
when $N=2$ there may exist 
zero, one, two, or even three bubble profiles, depending on the values of
$\alpha$ and $\sigma$, while if $N \geq 3$, 
there can be no more than two bubble profiles. 
The $\Gamma$-limit of higher dimensional traveling waves of \eqref{main11} is a future work in progress. \\    

{We give an outline for the remainder of the paper. 
In Section \ref{bv} we introduce a class of functions which are in the spaces of weighted bounded variation. They fit in with the $\Gamma$-convergence for the traveling wave problem.
The variational formulation 
for the FitzHugh-Nagumo system is given in Section \ref{var} and conditions 
for the $\Gamma$-convergence are verified 
 in Section \ref{conv}.  
 We then in Sections \ref{min}
demonstrate that the limiting functional has a minimizer that corresponds to traveling front or pulse. Section \ref{fr} and Section \ref{pul} will 
 distinguish front from pulse under different physical parameter regimes.
Moreover these front and pulse in the limiting case imply the existence of the traveling waves 
if $\epsilon$ is sufficiently small. }

\section{Weighted BV function} \label{bv}
\setcounter{equation}{0}
\renewcommand{\theequation}{\thesection.\arabic{equation}}
{If $(x,y) \in \Omega = {\mathbb R} \times {\widetilde{\Omega}}$ for 
a bounded smooth domain {$\widetilde{\Omega}$}  in ${\mathbb R}^{n-1}$, 
we write $|(x,y)|=\sqrt{x^2+ \sum_{i=1}^{n-1} y_i^2}$ to denote 
the usual Euclidean norm. Let $L^p_e =\{w: \int_\Omega  e^x |w|^p \,dx \, dy < \infty \}$ for $p \geq 1$.} \\ 



\noindent
{\em Definition:} Given a function $u \in L^1_{loc}(\Omega)$ the total variation of $u$ in $\Omega$  with respect to  
the measure $e^xdx\, dy$ is defined as
\begin{equation} \label{def_bv}
\|Du\|_e(\Omega)=\sup \{ \int_\Omega u \, \mbox{div} \, (e^x \varphi ) \, dx\, dy:  \; \varphi \in C^1_0(\Omega; {\mathbb R}^n), |\varphi| \leq 1 \} \;.
\end{equation}
If $\|Du\|_e(\Omega)<\infty$, it is easily checked \cite{B} that  for  any open set $\Omega'\subset\!\subset\Omega$ the total variation of $u$ in $\Omega'$
$$
\|Du\|(\Omega')=\sup \{ \int_{\Omega'} u \, \mbox{div} \, \varphi  \, dx\, dy:  \; \varphi \in C^1_0(\Omega'; {\mathbb R}^n), |\varphi| \leq 1 \}
$$
is also finite. Therefore if $u \in L^1_{loc}(\Omega)$ and $\|Du\|_e(\Omega)<\infty$, {then} the function $u$ has bounded variation {on} any $\Omega'\subset\!\subset\Omega$. By the Riesz representation theorem there exist a Radon measure $\mu$ and a $\mu$-measurable function $\sigma$ with $|\sigma|=1$ such that for any  $\varphi\in C^1_0(\Omega;\mathbb R^n)$
\begin{equation}\label{mia}
 \int_\Omega u \, \mbox{div} \, \varphi  \, dx\, dy = - \int_\Omega \varphi \cdot \sigma \, d \mu\;.
\end{equation}
It follows from \eqref{mia} that the vector valued measure $\sigma d\mu$ coincides with the distributional gradient $Du$ of $u$. Hence, if  $u$ is smooth
then {$\sigma d\mu=  \nabla u \,dx \,dy$ and $d \mu=|\nabla u| \,dx \, dy$.}

Note that \eqref{mia} in particular {implies} that  for any  $\varphi\in C^1_0(\Omega;\mathbb R^n)$
\begin{equation} \label{def_mu}
 \int_\Omega u \, \mbox{div} \, (e^x \varphi ) \, dx\, dy = - \int_\Omega e^x \varphi \cdot \sigma \, d \mu \;.
\end{equation}
For $\varphi=(\varphi_1, \varphi_2, \dots, \varphi_n)$, it is clear that 
 $e^x \, \mbox{div} \,\varphi=
\mbox{div} \,(e^x \varphi)- e^x \varphi_1$ and thus
\begin{equation} \label{by_parts}
\int_\Omega e^x u \, \mbox{div} \, \varphi \, dx \, dy= - \int_\Omega (e^x u \varphi_1 \, dx \, dy + e^x\varphi \cdot \sigma \, d\mu) \;,
\end{equation}
a formula that we will use later. 

 We shall denote by $BV_e(\Omega)$ the set of functions $u\in L^1_e(\Omega)$ such that $\|Du\|_e(\Omega)<\infty$. {This function space} is equipped with the norm
 \[
 \| u \|_{BV_e}= \| Du \|_{e} (\Omega) + \|u\|_{L^1_e}\;.
 \]
Also, if $u\in BV_e(\Omega)$ then
\begin{equation}\label{form}
\|Du\|_e(\Omega)=\int_\Omega e^x\,d\mu\;.
\end{equation}
{The first of the next two lemmas is a straightforward consequence of 
 \eqref{def_bv}, {while} the second can be proved exactly the same as the standard 
 $BV$ functions;} see  \cite[p.172, Theorem 2]{EG}.  
\begin{lemma} (lower semicontinuity of weighted variation measure) \label{lem_semi} \\
Suppose $\{ u_k \} \subset BV_e$
and $u_k \to u_0$ in $L^1_{e, \, loc}$, then 
$\| Du_0 \|_e (\Omega)  \leq {\liminf_{k\to\infty}} \| Du_k \|_e (\Omega)$.
\end{lemma}
\begin{lemma} (local approximation by smooth functions) \label{lem_approx} \\
Suppose $u \in BV_e$. Then there exists function $\{ u_k \} \subset BV_e \cap C^{\infty}$ such that\\
(i) $u_k \to u$ in $L^1_e$ and \\
(ii) $\| Du_k \|_e(\Omega) \to \|Du \|_e(\Omega)$ as $k \to \infty$.
\end{lemma}

Our attention will be focused on the one dimensional case {from now on}; that is, 
 $\Omega=(-\infty,\infty)$. Suppose $u=\chi_{[a,b]}$ with $-\infty \leq a <b < \infty$,
an easy calculation gives
\begin{equation}\label{fac}
\| D\chi_{[a,b]} \|_e(\Omega)= \sup_{|\varphi|\leq1} \int_a^b (e^x \varphi)' \, dx=
 \sup_{|\varphi|\leq1} \{ e^x \varphi |_{x=a}^b \}
 = e^b+ e^a \;.
\end{equation}
The following lemma will be useful in the sequel. 
\begin{lemma} \label{lem_converse}
If $E= \cup_i [a_i,b_i]$ is a union of countably many disjoint intervals, then
\begin{equation} \label{Du}
\|D \chi_E \|_e (\Omega)= \sum_i (e^{a_i}+e^{b_i}) \;.
\end{equation}
Moreover $\chi_E \in BV_e$ if and  only if $E$ is the union of countably many disjoint intervals $[a_i,b_i]$ and the right hand side of \eqref{Du} is finite.
\end{lemma}
\begin{proof}
It is clear that \eqref{Du} is an immediate consequence of \eqref{fac}. Furthermore, if  $E= \cup_i [a_i,b_i]$ is the union of countably many disjoint intervals and the right hand side of \eqref{Du} is finite, then 
$$
\int_{-\infty}^{+\infty}e^x\chi_E(x)\,dx\leq \sum_ie^{b_i}<\infty.
$$
This implies {$\chi_E \in L^1_e$} and thus $\chi_E \in BV_e$. \\

Conversely, if $\chi_E\in BV_e$, then, as observed before, $\chi_E$ has finite total variation in any bounded interval. It follows that, (see, e.g. \cite[Proposition 3.52]{AFP}), $E$ is the union of countably many disjoint intervals $[a_i,b_i]$ and the conclusion {is immediate} from \eqref{Du}.
\end{proof}

%

\begin{lemma} \label{lem_shift}
Let $u \in BV_e$ and $h \in {\mathbb R}$. Then \\
(i) $\| u(\cdot+h) \|_{L^1_e}=e^{-h} \|u \|_{L^1_e}$; \\
(ii) 
{ $\| D (u(\cdot+h)) \|_e (\Omega) = e^{-h} \| Du \|_e (\Omega)$}; \\
 (iii) $u^+$ and $u^-$ are in $BV_e$ with $\| Du^+ \|_e \leq \| Du \|_e$ and $\| Du^- \|_e \leq \| Du \|_e$;\\
 (iv) $\| u \|_{L^1_e} \leq  \| Du \|_e (\Omega)$; \\
 (v) $|u(x)| e^x \leq  2 \| Du \|_e(\Omega)$ a.e..
\end{lemma}
\begin{proof}
(i) to  (ii) are directly follow from 
 the definitions. \\

If $u\in BV_e$ then $u$ has locally finite variation in $\mathbb R$. This implies that $u^+$ has locally finite variation in $\mathbb R$ and $d(Du^+)=\chi_{\{u>0\}}d(Du)=\chi_{\{u>0\}}\sigma d\mu$ (see \cite[Example 3.100]{AFP}). It follows from \eqref{form} that 
$$
\|Du^+\|_e(\Omega)=\int_{\{u>0\}} e^x\,d\mu\leq \|Du\|_e(\Omega)\;.
$$
 Similar calculations work for $u^-$, so
%
 (iii) follows. \\

By Lemma~\ref{lem_approx} it suffices to verify 
(iv) and (v) for smooth $u$ only. 
Consider the case $u \geq 0$ first. For any $\delta>0$, 
 there exists a large $R>0$ such that
$0 \leq \int_{|x| > R} e^x u \,dx \leq \delta$. 
 Take 
a $\varphi \in C^{\infty}_0({\mathbb R})$ such that  $\varphi=1$
on $[-R,R]$, $\varphi =0$ for $|x| \geq R+2$, and $0 \leq \varphi \leq 1$ on $[R,R+2]$ with 
$|\varphi'| \leq 1$. Then
\begin{eqnarray*}
\| u \|_{L^1_e} -\delta & \leq & \int_{-\infty}^\infty e^x u \varphi \,dx = - \int_{-\infty}^\infty e^x\varphi \cdot \sigma \, d \mu
 - \int_{-\infty}^{\infty} e^x u \, \varphi' \, dx \\
 & \leq & \| Du \|_e (\Omega) +  \int_{R \leq |x| \leq R+2} e^x |u|  \, dx \\
 & \leq & \| Du \|_e (\Omega)+ \delta \;.
 \end{eqnarray*}
This gives $\| u \|_{L^1_e} \leq \| Du \|_e (\Omega)$ if $u \geq 0$. 
 In the general case, setting 
 $u=u^+-u^-$ yields 
 \begin{eqnarray*}
 \| u \|_{L^1_e} & =  &\| u^+ \|_{L^1_e} + \| u^- \|_{L^1_e}  \\
   & \leq & \| Du^+ \|_{L^1_e} + \| Du^-\|_{L^1_e} \\
   & = & \| Du \|_{L^1_e} \\
   &\leq & \| Du \|_e (\Omega) 
 \end{eqnarray*}
which completes the proof of (iv). \\
 
 Suppose $u$ is smooth and {with compact support}. Then ${-u e^x} =\int_x^\infty (ue^t)' \, dt= \int_x^\infty e^t (u+u') \, dt$. Therefore 
 \[
 |u(x)| e^x \leq \int_{x}^\infty e^t (|u|+|u'|) \, dt \leq \| u \|_{L^1_e} + \|Du \|_e (\Omega) \leq 2 \| Du \|_e (\Omega) \;,
 \]
 which yields (v) in this special case. {This inequality can now be extended to $u \in C^{\infty}$ by using $u \varphi$ as an approximation, 
 where $\varphi$ is the cut-off function that we employed earlier.}
\end{proof}
%
%
 

\section{{Variational formulation}} \label{var} 
\setcounter{equation}{0}
\renewcommand{\theequation}{\thesection.\arabic{equation}}

We now turn to the variational formulation for studying the traveling waves of (\ref{main11}). 
Let 
 $H^1_e=\{ w: \int_{-\infty}^\infty e^x (w'^{\,2}+w^2) \,dx < \infty \}$ and {$F_\epsilon(w)  \equiv  - \int_0^{w} f_\epsilon(\xi)  \; d\xi $.} 
For a given $u$, we designate the unique solution of 
 (\ref{dc2}) 
by $v={\cal L}_c u$; {here} 
${\cal L}_c: L^2_e \to H^1_e$ {is} a self-adjoint operator with respect to the inner product on $L^2_e$.
For given $c, d > 0$, let ${\cal I}_{c,d}: H^1_e \to \mathbb{R}$ defined by 
\begin{equation} \label{jc}
{{\cal I}_{c,d}}(w)= \int_{-\infty}^\infty e^x (\frac{dc^2}{2} w'^{\,2}+ F_\epsilon(w) + \frac{\epsilon \sigma}{2} w {\cal L}_{c} w ) \, dx \;.
\end{equation}
{The standard variational argument shows that $(u,v,c)$ solves (\ref{dc1})-(\ref{dc2})  
 provided $v={\cal L}_{c} u$ and 
 $u$ is a critical point of ${\cal I}_{c,d}$.
 For easy referral, the terms on the right of (\ref{jc}) are called the {gradient energy, the $F$-integral} (or potential)} and 
the nonlocal energy, respectively. The nonlocal energy is always non-negative since
 $\int_{-\infty}^{\infty} e^x w \, {\cal L}_{c} w \, dx \geq 0$ for any $w \in L^2_e$. \\

Given an $\epsilon_1>0$ such that $\max\{ \frac{1}{\sigma}, \frac{1}{2 \alpha} \} \geq \epsilon_1>0$. 
For all 
$\epsilon \in (0,\epsilon_1)$, 
there exist 
$\beta_2>1$ with $\beta_2-1$ being small and
$\widetilde{M_1}=\widetilde{M_1}(\gamma) >0$ satisfying {$f_\epsilon(-\widetilde{M_1}) 
\geq \frac{\beta_2 }{\gamma}$.}
We consider ${\cal I}_{c,d}:Y \to {\mathbb R}$ with $Y$ being a admissible set defined by 
\[
Y \equiv \{ w \in H^1_e: \; \int_{-\infty}^\infty e^x w^2 \, dx=1, \; - \widetilde{M_1}-1 \leq w \leq \beta_2 \}. 
\]
The constraint 
$\| w \|_{L^2_e}=1$ imposed in $Y$ is
to eliminate a continuum of critical points due to translation. Suppose $u \in Y$ is a constrained minimizer of ${\cal I}_{c,d}$, 
 it is the sought-after traveling wave solution provided ${\cal I}_{c,d}(u)=0$ and 
 $-\widetilde{M_1}-1<u<\beta_2$. We refer to \cite{CC2} for the detailed argument.\\ 

{The constrained variational approach has been employed \cite{CCH,CC2,CC3} to establish the existence of traveling wave solutions of  FitzHugh-Nagumo system. There all the parameters are fixed and of order $O(1)$, except that $d$ can be sufficiently small. In the  situation as $d \to 0$ the wave speed $c$ tends to infinity and $dc^2$ approaches to a positive number, which depends {only} on $\beta$ if $f(\xi) =\xi(\xi-\beta)(1-\xi)$ with $\beta \in (0,1/2)$. This is a case that {the} $\Gamma$-limit {of ${\cal I}_{c,d}$} does not exist when $d \to 0$. It is interesting to {investigate} if the tool of 
$\Gamma$-convergence can be utilized to {study traveling waves;  
this will, for the first time, expand its scope of applicability to non-stationary problems. In so doing we require} other parameters {to} change in some coordinate fashion with $d$.}  

 \begin{remark} \label{remark_sec3}
 The proofs given in \cite{CCH,CC2, CC3} used a different constraint $\|w' \|_{L^2_e}=\sqrt{2}$; 
it will be seen that 
imposing $\|w \|_{L^2_e}=1$ in $Y$ 
 facilitates {easier $\Gamma$-convergence analysis.} 
\end{remark}


{The $\Gamma$-limit of ${\cal I}_{c,d}$ and its minimizer will be studied in the later sections. The next two lemmas enable us to 
recover the traveling wave solutions 
 from the minimizer of the limiting functional of ${\cal I}_{c,d}$.} 

\begin{lemma} \label{lem_minimum}
{Let $\gamma$, $\sigma$, $\alpha$ and $\epsilon_1$ be given. If $c,d>0$ and $\epsilon \in(0,\epsilon_1)$ then $\inf _{w \in Y} {\cal I}_{c,d}(w)$ is,
{uniformly in $\epsilon$,
bounded from below. }
Suppose, in addition,
$\inf _{w \in Y} {\cal I}_{c,d}(w) \leq 0$, 
then
there is a minimizer $u \in Y$.}
\end{lemma}
\begin{proof}
{There exists a constant $M_2>0$
such that $-M_2 \xi^2 \leq {F_\epsilon(\xi)}$ for all $\xi \in {\mathbb R}$. 
For $w \in Y$,}
\[
{\cal I}_{c,d}(w) \geq \int_{-\infty}^\infty e^x F_\epsilon(w) \,dx \geq -M_2 \int_{-\infty}^\infty e^x w^2 \,dx = -M_2, 
\]
which shows $\inf _{w \in Y} {\cal I}_{c,d}(w)$ is, {uniformly in $\epsilon$,} bounded from below. Taking 
 a minimizing sequence $\{ w_n \} \subset Y$ with ${\cal I}_{c,d}(w_n) \leq  \inf_{w \in Y} {\cal I}_{c,d}(w) +1$, {then}
\begin{eqnarray*}
\frac{dc^2}{2} \int_{-\infty}^\infty e^x w'_n{^{2}}\, dx & \leq & {\cal I}_{c,d}(w_n) - \int_{-\infty}^\infty e^x F_\epsilon(w_n) \, dx \\
& \leq &  \inf_{w \in Y} {\cal I}_{c,d}(w) +1 + M_2 \;.
\end{eqnarray*}
Recall a Poincare type inequality for $w \in H^1_e$:
 \begin{eqnarray}
 \int_{\mathbb R} e^x w'^{\,2} \, dx \geq \frac{1}{4} \int_{\mathbb R} e^x w^2 \, dx.
\end{eqnarray}
This gives a uniform bound for $\| w_n \|_{H^1_e}$ {for all $n$}. Along a subsequence there is a $W \in H^1_e$ such that 
$w_n \rightharpoonup  W$ weakly in $H^1_e$ and strongly in $L^{\infty}_{loc}(\mathbb{R}) \cap L^2_{loc}({\mathbb R})$.
As in the proof of Lemma 4.2 in \cite{CC2}, we obtain $-\widetilde{M_1}-1 \leq W \leq \beta_2$, 
\begin{eqnarray*}
\int_{\mathbb R}e^x W'^{\,2}dx & \leq & \lim\inf \int_{\mathbb R} e^x w_n'^{\,2} dx, \\
\int_{\mathbb R} e^x F_\epsilon(W) dx & \leq & \lim \inf \int_{\mathbb R} e^x F_\epsilon(w_n) dx, \\
\int_{\mathbb R} e^x W {\cal L}_c W \, dx  &\leq & \lim \inf \int_{\mathbb R} e^x w_n {\cal L}_c w_n \, dx \;,
\end{eqnarray*}
and thus ${\cal I}_{c,d}(W) \leq {\lim\inf}_{n \to \infty} {\cal I}_{c,d}(w_n)$. 
  Moreover 
 $\int_{\mathbb R} e^x W^2 \, dx \leq 1$, since $\int_{-l}^l e^x W^2 \, dx = \lim_{n \to \infty} \int_{-l}^l e^x w_n^2 \, dx \leq 1$ for any $l>0$.
 
Suppose 
$\inf_{w \in Y} {\cal I}_{c,d}(w) \leq 0$ then ${\cal I}_{c,d}(W) \leq 0$.
We claim $W \not \equiv 0$; 
for otherwise
\begin{eqnarray*}
0& \geq & \lim_{n\to \infty} \inf {\cal I}_{c,d}(w_n)\\
&\geq & \frac{d c^2}{8} \int_{\mathbb R} e^x w_n^2 \, dx + \lim_{n\to \infty} \inf
\int_{\mathbb R} e^{x} \{  F_\epsilon(w_n) + \frac{1}{2}w_n  \, {\cal L}_{c} w_n \} \, dx \\
&\geq& \frac{d c^2}{8}+ \int_{\mathbb R} e^{x} \{ F_\epsilon(W) + \frac{1}{2} W \, {\cal L}_{c} W \} \, dx \\
&=& \frac{d c^2}{8} \;,
\end{eqnarray*}
which is absurd.
Consequently 
$1 \geq \int_{\mathbb R} e^{x} W^2 \,dx >0$. \\ 

Take $a \geq 0$ such that $e^a \int_{\mathbb R} e^x W^2 \,dx=1$.  
Letting $u(x) \equiv W(x-a)$ gives 
$u \in Y$ and 
\[
{\cal I}_{c,d}(u) = e^a {\cal I}_{c,d}(W) \leq {\cal I}_{c,d}(W) \leq \inf_{w \in Y} {\cal I}_{c,d}(w) \leq {\cal I}_{c,d}(u) \;.
\]
Hence $u$ is a minimizer of ${\cal I}_{c,d}$. 
 In case $\inf_{w \in Y} {\cal I}_{c,d}(w)<0$ then $a=0$ and $u=W$.
\end{proof}

\begin{lemma} \label{lem_soln}
Suppose 
$u \in Y$ is a minimizer of ${\cal I}_{c,d}$ 
and ${\cal I}_{c,d}(u)=0$, then  
{$(u,{\cal L}_c u)$} is a traveling wave solution of (\ref{dc1})-(\ref{dc2}) with $c$ as {its} wave speed.
\end{lemma}
\begin{proof}
The first step is to show $-\widetilde{M_1}-1<u<\beta_2$, using the arguments of \cite{CCH,CC2}. 
{Since ${\cal I}_{c,d}(u)=0$, 
we may slightly modify the proofs given in 
\cite{CC2,CCH} 
 to get rid of 
 the Lagrange multiplier associated with the constraint $\int_{\mathbb R} e^x w^2 \,dx=1$. 
This ensures that $(u,{\cal L}_c u)$ is a traveling wave solution, because $u$ acts like an unconstrained critical point of ${\cal I}_{c,d}$.}
\end{proof}




\section{$\Gamma$-convergence} \label{conv}
\setcounter{equation}{0}
\renewcommand{\theequation}{\thesection.\arabic{equation}}
To investigate the
 $\Gamma$-convergence for 
 the traveling wave {functional}, we rewrite \eqref{dc1}-\eqref{dc2} in the following form:
  \begin{eqnarray}
\epsilon^2 u_{xx} + \epsilon^2 u_x  +{f_\epsilon}(u)- \epsilon \sigma v& = &0 \;, \label{FNTW1} \\ 
c^2 v_{xx} +c^2 v_x  -\gamma v +u &=&0 \; ; \label{FNTW2}
\end{eqnarray}
that is, set $\epsilon = \sqrt {dc^2} \,$. 
 It is required that $d$ be related to $\epsilon$ in some suitable fashion to induce an interesting geometric variational problem to be 
 the $\Gamma$-limit. Thus at least one of $c,d$ is 
depending on $\epsilon$ when we consider a sequence along $\epsilon \to 0$. \\

Note that $F_\epsilon = F_0+ {\alpha \epsilon}G $, 
where $F_0(u) \equiv \frac{1}{4} u^2 (u-1)^2 $ and $G(u) \equiv \frac{1}{\sqrt{2}}( \frac{u^3}{3}  -\frac{u^2}{2})$. 
In the decomposition of $F_\epsilon$, $F_0$ is a balanced bistable nonlinearity and $ F_0(0)=F_0(1)=\min F_0 =0$.
$G$ has a local maximum at $0$ and a local minimum at $1$ with $G(0)=0$ and
$G(1)=-1/6 \sqrt{2}$. Their combined $F_\epsilon$ has a local maximum 
at $\beta_\epsilon$ and $F_\epsilon(\beta_\epsilon)>0$.
 ${F_\epsilon}(0)=0$ is a local minimum and $F_\epsilon(1)=-\frac{1-2\beta_\epsilon}{12}=-\frac{1}{6\sqrt{2}} \alpha \epsilon$ is the global minimum. \\
 
To evaluate the $\Gamma$-limit of the functional {${\cal I}_{c,d}/\epsilon$} in the topology $L^2_e$ as $\epsilon \to 0$, we set $J_{c(\epsilon)} \equiv {\cal I}_{c,d}/\epsilon$; that is,
\begin{eqnarray}
J_{c(\epsilon)}(w)
&=&  \int_{-\infty}^\infty e^x \{  \frac{\epsilon w'^{\,2}}{2} + \frac{F_0(w)}{\epsilon} + {\alpha} G(w) 
+ \frac{\sigma}{2} w {\cal L}_{c(\epsilon)} w \}\,dx , \label{J1}
\end{eqnarray}
for $w \in Y$. In \eqref{J1}, $c$ is not {necessarily} a constant but a function of $\epsilon$ with the property that $c(\epsilon) \to c_0$ 
{for some positive constant $c_0$} as $\epsilon \to 0$. If $u$ 
 is a minimizer of $J_{c(\epsilon)}$ in $Y$ and 
\begin{equation} \label{J0}
J_{c(\epsilon)}(u)=0 \; 
\end{equation}
then by Lemma~
\ref{lem_soln}, we obtain a solution of \eqref{FNTW1}-\eqref{FNTW2} by setting $v={\cal L}_{c(\epsilon)} u$. 
Our goal aims {for} traveling wave with speed 
  close to $c_0$. A traveling wave solution of \eqref{FNTW1}-\eqref{FNTW2} will be denoted by $(c_\epsilon,u_\epsilon,v_\epsilon)$, where $c_\epsilon$ is the wave speed and $d=\epsilon^2/c^2_\epsilon$. \\

Next we 
 examine the $\Gamma$-convergence of $J_{c(\epsilon)}$ as $\epsilon \to 0$. As an application of the {limiting functional later on},  
 we prove the existence of 
  traveling wave solutions {of the original problems by seeking} suitable conditions 
 on the parameters $\alpha$, {$\gamma$} and $\sigma$. 
 Let 
  $\phi(\xi)=\int_0^\xi \sqrt{2 F_0(\eta)} \,d\eta$. This function has been frequently used \cite{KS,M} in the calculation of 
 phase transition problems. 
Although {these latter problems bear} certain similarity to 
 our study, additional complexities arise
due to the treatment for 
 unbounded domains,  the presence of a nonlocal term, and {the appearance of 
  the weight
$e^x$ in dealing with traveling waves instead of stationary solutions.} As a remark, $\phi$ is a 
 strictly increasing function, 
 $\phi(0)=0$ and $\phi(1)=\frac{\sqrt{2}}{12}$.  \\ 

{In what follows, when we say a sequence converges, it might be passing 
 through 
 a 
  subsequence without further comment. 
Denoted by $C_k, k=0,1,2, 
\cdots,$} a positive constant 
 not depending on $\epsilon$. 
For instance, let 
$C_1 \equiv \frac{1}{\sqrt{2}} (\frac{1+ \widetilde{M_1}}{3}+\frac{1}{2})$ in the following compactness lemma. 

\begin{lemma} (compactness) \label{lem_compact} 
Let $\{ w_\epsilon \} \subset Y$ such that $\lim \inf_{\epsilon \to 0} J_{c(\epsilon)}(w_\epsilon) 
 \leq C_0$. 
Then \\
(i) $\phi(w_\epsilon) \in BV_e$; \\ 
(ii) there exists a subsequence, still denoted 
 by $\{ w_\epsilon \}$, and a characteristic function
$\chi_E \in BV_e \cap L^2_e$ such that 
$w_\epsilon \to \chi_E$ in $L^2_e$ and $\| D \chi_E \|_e (\Omega) \leq 6 \sqrt{2} (C_0+C_1 \alpha)$;\\
(iii) $\int_{-\infty}^\infty e^x \chi_E =1$ 
 with $E \subset (-\infty, \log 6 \sqrt{2} (C_0+C_1 \alpha) \,]$.
\end{lemma}

%
\begin{proof}
For $\{ w_\epsilon \} \subset Y$, we know $\| w_\epsilon \|_{L^\infty} \leq 1+ \widetilde{M_1}$ and     
$\| w_\epsilon \|_{L^p_e} \leq \| w_\epsilon \|_{L^2_e}^{2/p} \; \| w_\epsilon \|_{L^{\infty}}^{(p-2)/p} \leq (1+ \widetilde{M_1})^{(p-2)/p}$ if $p>2$. 
Since $w_\epsilon \in H^1_e$ 
 and the nonlocal energy is non-negative,
\begin{eqnarray}
\| D (\phi(w_\epsilon)) \|_e (\Omega) 
&= &  \int_{-\infty}^\infty e^x |\sqrt{2 F_0(w_\epsilon)}| \; | w'_\epsilon | \, dx \nonumber \\
& \leq & \int_{-\infty}^\infty e^x (\epsilon \frac{w'^{\,2}_\epsilon}{2} +\frac{F_0(w_\epsilon)}{\epsilon}) \, dx \nonumber \\
& \leq & J_{c(\epsilon)} (w_\epsilon) +  \alpha \int_{-\infty}^{\infty} e^x |G(w_\epsilon)| \, dx \nonumber \\
& \leq & C_0+ \frac{\alpha}{\sqrt{2}} (\frac{1}{3} \int_{-\infty}^{\infty} e^x |w_\epsilon|^3 \, dx
 + \frac{1}{2} \int_{-\infty}^\infty e^x |w_\epsilon|^2 \, dx ) + {o(1)} \nonumber \\
 & \leq & C_0+\frac{\alpha}{\sqrt{2}} (\frac{1+ \widetilde{M_1}}{3}+\frac{1}{2}) + o(1)= C_0+C_1 \alpha + {o(1)}\;, \label{dphi}
\end{eqnarray}
which shows $\phi(w_\epsilon) \in BV_e$ with a uniform bound in norm. On each finite interval $[-k,k]$, 
applying the compactness 
theorem for bounded variation functions on finite domains, we obtain 
 a function $\Phi_0 \in L^1_e(-k,k)$ such that along a subsequence $\phi(w_\epsilon) \to \Phi_0$ in $L^1_e(-k,k)$ and pointwise a.e.. Then using 
a diagonal process, we conclude that $\Phi_0 \in L^1_{e, \,loc}({\mathbb R})$ and through a subsequence 
$\phi(w_\epsilon) \to \Phi_0$ in $L^1_{e,\,loc}$ and pointwise a.e.. 
By Lemma~\ref{lem_semi} and Lemma~\ref{lem_shift}, we know $\Phi_0 \in BV_e$ and 
\[
\| \Phi_0 \|_{L^1_e} \leq \| D \Phi_0 \|_e (\Omega) \leq \lim \inf_{\epsilon \to 0} \| D \phi(w_\epsilon) \|_e (\Omega) \leq C_0+C_1 \alpha. 
\]
Since $\phi$ is a strictly increasing function, setting $w_0 \equiv \phi^{-1}(\Phi_0)$ yields $w_\epsilon \to w_0
$ pointwise a.e. and thus 
$-\widetilde{M_1}-1 \leq {w_0} \leq \beta_2$ a.e..
Observe that 
\begin{equation} \label{F0}
 \int_{-\infty}^\infty e^x F_0(w_\epsilon)\, dx 
\leq   \epsilon \{ J_{c(\epsilon)} (w_\epsilon) +  \alpha \int_{-\infty}^\infty e^x 
|G(w_\epsilon) | \, dx \,  \}
\leq  \epsilon (C_0+C_1 \alpha + o(1)) \;.
\end{equation}
Applying Fatou's lemma gives 
\begin{eqnarray*}
0 \leq \int_{-\infty}^\infty e^x F_0(w_0) \,dx &\leq  & \lim \inf_{\epsilon \to 0} \int_{-\infty}^\infty e^x F_0(w_\epsilon)\, dx  \nonumber \\
& \leq & \lim \inf_{\epsilon \to 0} \epsilon (C_0+C_1 \alpha + o(1)) \nonumber \\
&=& 0 \;,
\end{eqnarray*}
 which implies $F_0(w_0)=0$ a.e.. This can be valid only if $w_0(x) \in \{ 0,1 \}$ 
  a.e.; in other words, $w_0=\chi_E$
for some set $E \subset {\mathbb R}$. As a consequence, $\phi(w_\epsilon) \to \Phi_{0}=\phi(w_0)= \phi(1) \chi_E$
in $L^1_{e,\,loc}$ and pointwise a.e.. 
Using Lemma~\ref{lem_semi} and Lemma~\ref{lem_shift} 
 again, we get 
\begin{equation} \label{bound1}
\phi(1) \| \chi_E \|_{L^1_e} \leq 
\phi(1) \| D \chi_E \|_e(\Omega) = \| D \phi(w_0) \|_e(\Omega) \leq \lim \inf \| D \phi(w_\epsilon) \|_e(\Omega) \leq 
C_0+C_1 \alpha. 
\end{equation}
Hence $\chi_E \in BV_e$ and $\| D \chi_E \|_e (\Omega) \leq \frac{C_0+C_1 \alpha}{\phi(1)}$, which shows 
$E \subset (-\infty,  \log \frac{C_0+C_1 \alpha}{\phi(1)} \,]$. \\ 

Invoking Lemma~\ref{lem_shift} yields $e^x |\phi(w_\epsilon(x))| \leq 2 \| D \phi(w_\epsilon) \|_e(\Omega) \leq 2 (C_0+C_1 \alpha {+o(1)})$ and 
thus $|\phi(w_\epsilon(x))| \leq 2 (C_0+C_1 \alpha {+1}) e^{-x}$. 
Since $\phi^{-1}$ is continuous and $\phi^{-1}(0)=0$, 
there exists a $y_0$ such that  
 $|w_\epsilon(x)| \leq 1/2$ if $x \geq y_0$. This together with (\ref{F0}) gives 
 \begin{eqnarray*}
 \int_{y_0}^\infty e^x w_\epsilon^2 \, dx \leq 4 \int_{y_0}^\infty e^x  w_\epsilon^2 (w_\epsilon -1)^2 \,dx
 \leq 16 \int_{-\infty}^\infty e^x F_0(w_\epsilon) \, dx \leq 16 \epsilon (C_0+C_1 \alpha {+o(1)}).
 \end{eqnarray*}

Next we 
 extract 
 a subsequence of $\{ w_\epsilon \}$ to form a Cauchy sequence in $L^2_e$. 
For any $\delta>0$, in view of 
\begin{eqnarray*}
&&\int_{-\infty}^{\infty} e^x |w_\epsilon  - w_\eta|^2 \, dx \\
& \leq  &\int_{-\infty}^{-y_\delta} e^x |w_\epsilon  - w_\eta|^2 \, dx 
+ \int_{-y_\delta}^{y_0} e^x |w_\epsilon  - w_\eta|^2 \, dx  + \int_{y_0}^{\infty} e^x |w_\epsilon  - w_\eta|^2 \, dx  \\
& \leq & 4 (1+ \widetilde{M_1})^2 \int_{-\infty}^{-y_\delta} e^x \,dx + \int_{-y_\delta}^{y_0} e^x |w_\epsilon  - w_\eta|^2 \, dx  
+ 2 \int_{y_0}^{\infty} e^x (w_\epsilon^2  + w_\eta^2 ) \, dx  \\
& \leq & 4 (1+ \widetilde{M_1})^2 e^{-y_\delta}  + \int_{-y_\delta}^{y_0} e^x |w_\epsilon  - w_\eta|^2 \, dx  
+ 64 (C_0+C_1 \alpha {+o(1)}) \, \max\{\epsilon,\eta\} \;,
\end{eqnarray*}
in the last line  
 the first term 
and the third term are smaller than $\delta$ if we pick 
 $y_\delta$ large and $\max \{\epsilon,\eta\}$ small enough. 
 With $w_\epsilon$ being
converging to $w_0$ pointwise a.e. and $\| w_\epsilon \|_{L^{\infty}} \leq 1+ \widetilde{M_1}$, the second term is also smaller than $\delta$ since 
the dominated convergence theorem implies that 
$w_\epsilon \to w_0$ in $L^2_e(-y_\delta,q_0)$.
Along this Cauchy subsequence, $\{ w_\epsilon \}$ 
 converges to a function 
  in $L^2_e$ and pointwise a.e. and consequently this limit function
   has to be $\chi_E$. Hence 
\[
\int_{-\infty}^{\infty} e^x \chi_E \, dx = \int_{-\infty}^\infty e^x \chi_E^2 \, dx = \lim \int_{-\infty}^\infty e^x w_\epsilon^2 \, dx
=1\;.
\]
The proof 
 is complete.
\end{proof}


To employ 
$\Gamma$-convergence 
 for studying the existence and qualitative behavior of  traveling wave solutions as $\epsilon \to 0$, 
  we extend the domain of $J_{c(\epsilon)}$ to $L^2_e$ by setting 
\begin{equation}\label{jce}
J_{c(\epsilon)}(w) 
= \left\{ \begin{array}{ll} \mbox{as (\ref{J1})},  & \mbox{if}\; w \in Y \;, \\
  \infty, & \mbox{if} \;  w \in L^2_e \setminus Y  \; .
  \end{array} \right.
  \end{equation}
Next we propose a possible candidate for the $\Gamma$-limit of $J_{c(\epsilon)}$ 
Let $E= \cup_i [a_i,b_i]$, a union of countably many disjoint intervals; 
here $b_1>a_1>b_2>a_2 > \dots$ and $a_i \to -\infty$, $b_i \to -\infty$ as $i \to \infty$. 
We introduce a functional 
 $J^*_{c}: L^2_e \to {\mathbb R}$ defined by 
\begin{eqnarray}
J^*_{c}(w) & \equiv &
\left\{ \begin{array}{l}
{\frac{\sqrt{2}}{12}} \sum_i (e^{a_i}+e^{b_i}) {- \frac{\sqrt{2} \, \alpha}{12}} \sum_i (e^{b_i}-e^{a_i}) 
+\frac{\sigma}{2} \int_{-\infty}^{\infty} e^x \chi_E {\cal L}_{c} \chi_E \, dx \;, \\
\hspace{1.5in} \mbox{if} \; w=\chi_E \in BV_e \; \mbox{and} \; \int_{-\infty}^\infty e^x \chi_E \,dx=1, \label{Jinfty}  \\
\infty, \hspace{1.25in} 
 \mbox{otherwise} \;. 
\end{array} \right.
\end{eqnarray}


\begin{remark} \label{remark3}
Since $ \sum (e^{b_i}-e^{a_i}) \leq  e^{b_1}<\infty$, if $J^*_{c}(\chi_E)< \infty$ then 
both the first and the third terms of $J^*_{c}(\chi_E)$ 
 are positive and bounded from above. 
\end{remark}

From (\ref{FNTW2}), it follows from integration by parts that 
\[
c^2(\epsilon) \| ({\cal L}_{c(\epsilon)} w)' \|_{L^2_e}^2 + \gamma \| {\cal L}_{c(\epsilon)} w \|^2_{L^2_e} = \int_{-\infty}^\infty e^x w \, {\cal L}_{c(\epsilon)} w \, dx \;.
\]
{Since $c(\epsilon) \to c_0$}, it follows that 
$ \| {\cal L}_{c(\epsilon)} w \|_{H^1_e} \leq C_3 \| w \|_{L^2_e}$ for some positive constant $C_3$, 
 depending  
on $c_0$ and $\gamma$ only. 
Suppose 
 $w_\epsilon \to w_0$ in $L^2_e$, it is easily verified that 
 ${\cal L}_{c(\epsilon)} w_\epsilon \to {\cal L}_{c_0} w_0$ in $H^1_e$ and 
$\int_{-\infty}^\infty e^x w_\epsilon {\cal L}_{c(\epsilon)} w_\epsilon \, dx  \to  \int_{-\infty}^\infty e^x w_0 {\cal L}_{c_0} w_0 \, dx$ as $\epsilon \to 0$.
Using this fact, we apply a well-known stability theorem \cite[proposition 2.3]{Braides2} to establish the liminf inequality {as follows.} 


\begin{lemma} (liminf inequality) \label{lem_liminf} 
If $w_0 \in L^2_e$ then  
\begin{equation} \label{liminf}
J^*_{c_0}(w_0) \leq {\lim \inf}_{\epsilon \to 0} J_{c(\epsilon)} (w_\epsilon) \;
\end{equation}
for any sequence $w_\epsilon \to w_0$ in $L^2_e$.
\end{lemma}
\begin{proof}
It suffices to treat the case that ${\lim \inf} J_{c(\epsilon)} (w_\epsilon) 
<\infty$ and $\{ w_\epsilon \} \subset Y$; 
otherwise there is nothing to prove. 
Then $\| w_\epsilon \|_{L^2_e}=1$ and $-\widetilde{M_1}-1 \leq w_\epsilon \leq \beta_2$. 
It follows from Lemma~\ref{lem_compact} and Lemma~\ref{lem_converse} 
 that $w_0=\chi_E \in BV_e \cap L^2_e$,
$\int_{-\infty}^\infty e^x \chi_E \, dx=1$,
and $E=\cup_{i=1}^\infty [a_i,b_i]$, a union of countably many disjoint intervals.  

As mentioned earlier, 
 the nonlocal term can be ignored in checking the $\Gamma$-convergence; 
the same is true for the term $\alpha \int_{-\infty}^\infty e^x G(u) \,dx$
by the 
 stability theorem \cite[proposition 2.3]{Braides2}, since 
$w_\epsilon \to \chi_E$ in $L^2_e$ and a uniform
$L^{\infty}$ norm bound on $w_\epsilon$ imply 
$w_\epsilon \to \chi_E$ in $L^3_e$ as well. As a consequence,
$$\alpha \int_{-\infty}^\infty e^x G(w_\epsilon) \,dx \to \alpha \int_{-\infty}^\infty e^x G(\chi_E) \,dx = \alpha G(1) \sum_i (e^{b_i}-e^{a_i})
{= -\frac{\sqrt{2} \, \alpha}{12} \sum_i (e^{b_i}-e^{a_i})}.$$
Thus it remains 
 to verify the liminf inequality for the gradient term and the $F_0$-integral only, 
  and this is well known
except for the weight function $e^x$ appeared in the formulation. 

Extracting useful calculations from (\ref{dphi}) and (\ref{bound1}), we arrive at 
\begin{eqnarray*} 
{ \frac{\sqrt{2}}{12}} \sum_i (e^{a_i}+e^{b_i}) & =& \phi(1) \| D \chi_E \|_e (\Omega) \\
&  \leq & \lim \inf \| D \phi (w_\epsilon) \|_e \\
& \leq  & \lim \inf \int_{-\infty}^\infty e^x ( \frac{\epsilon w_\epsilon'^{\,2}}{2} +\frac{F_0(w_\epsilon)}{\epsilon}) \, dx \;,
\end{eqnarray*}
from which the proof of (\ref{liminf}) {can be completed}. 
\end{proof}

In studying the limsup inequality, the first step in the proof is to construct auxiliary functions as in dealing with 
 phase transition problems \cite{Leoni}.
{Let $f_0(\xi)=-\xi (\xi-1/2) (\xi -1)$.}
Integrating the equation 
\begin{equation} \label{f0}
\epsilon^2 U_\epsilon''+f_0(U_\epsilon)=0
\end{equation}
once yields
$$\epsilon^2 U'^{\,2}_\epsilon/2 - F_0(U_\epsilon)= \mbox{constant}.$$ 
Assigning this 
 constant to be $\epsilon/2$ gives 
\begin{equation} \label{U1}
U'_\epsilon= \frac{\sqrt{\epsilon+2 F_0(U_\epsilon)}}{\epsilon}
\end{equation}
with initial condition $U_\epsilon(0)=0$. This function {$U_\epsilon(x)$} is 
 strictly increasing and
\[
\int_0^{U_\epsilon} \frac{\epsilon}{\sqrt{\epsilon+2F_0(s)}} \, ds = x \;.
\]
It is now immediate that $U_\epsilon(\rho_\epsilon)=1$ and 
\begin{equation} \label{xU}
\rho_\epsilon= \int_0^1 \frac{\epsilon}{\sqrt{\epsilon+2F_0(s)}} \, ds \leq \sqrt{\epsilon} \;.
\end{equation}
By the same token, 
a strictly decreasing function $\widetilde{U_\epsilon}$ 
satisfies 
\begin{equation} \label{U2}
\widetilde{U_\epsilon}'= -\frac{\sqrt{\epsilon+2 F_0(\widetilde{U_\epsilon})}}{\epsilon}\;,
\end{equation}
$\widetilde{U_\epsilon}(0)=1$ and $\widetilde{U_\epsilon}(\rho_\epsilon)=0$. \\ 

We now state limsup inequality: \\ 

For every $w_0 \in L^2_e$, there exists a sequence $\{ w_\epsilon \} \subset L^2_e$
such that $w_\epsilon \to w_0$ in $L^2_e$ and $\lim \sup J_{c(\epsilon)} (w_\epsilon) \leq J^*_{c_0}(w_0)$. \\

\noindent
By \eqref{Jinfty} it suffices to consider 
   $J^*_{c_0}(w_0)<\infty$. Then 
$w_0= \chi_E \in BV_e$ for some $E \subset {\mathbb R}$, $\int_{-\infty}^\infty e^x \chi_E \,dx=1$ and 
  it is clear that 
   $\{ w_\epsilon \} \subset Y$. 
\begin{lemma} (limsup inequality) \label{lem_limsup} 
For every $\chi_E \in BV_e$ with $\int_{-\infty}^{\infty} e^x \chi_E \, dx=1$, there exists a sequence $\{ w_\epsilon \} \subset Y$
such that $w_\epsilon \to \chi_E$ in $L^2_e$ and $\lim \sup_{\epsilon \to 0} J_{c(\epsilon)}(w_\epsilon) \leq J^*_{c_0}(\chi_E)$.
\end{lemma}
\begin{proof}
Lemma~\ref{lem_converse} indicates that $\chi_E$ can jump at countably many points $x_1>x_2>x_3> \dots$
 with $x_n \to -\infty$. (If $E= \cup_j [a_j,b_j]$, then we set $x_1=b_1$, $x_2=a_1$ and $x_{2j-1}=b_j$, $x_{2j}=a_j$).
 For any $\delta>0$, choose a large ${i}_\delta$ such that $e^{-i_\delta} \leq \delta$ and thus
 $\int_{-\infty}^{-i_\delta} e^x \chi_E \, dx \leq \delta$.
 Without loss of generality, we may assume that $x_j \ne {-}i_\delta$ for all $j$. \\

  Let $x_j \in (-i_\delta,\infty)$ if $j=1,2, \dots, k_\delta$ and be outside of the interval $[-i_\delta,\infty)$ {otherwise}. 
 There is a $\Delta_{\delta}>0$ such that ${\min}_{1 \leq j \leq k_\delta} (x_{j}-x_{j+1}) \geq \Delta_\delta$.
 Let $\epsilon$ satisfy $\rho_\epsilon \leq \sqrt{\epsilon} < \Delta_\delta/2$. 
 It is clear that as $x$ increases,
 $\chi_E(x)$ jumps from $0$ to $1$ at $x_{2j}$  and 
 from $1$ to $0$ at $x_{2j-1}$.  Take $0 \leq \theta_i \leq 1$ for $i=1, 2, \dots, k_{\delta}$ and
define
 \[
w_\epsilon(x) = \left\{ \begin{array}{ll}
 U_\epsilon(x-{q}_{2j}),  &  x \in [q_{2j}, q_{2j}+ \rho_\epsilon], \; q_{2j}=x_{2j}-\theta_{2j} \rho_\epsilon,
  \; \mbox{when} \;2j \leq k_\delta, \\
 \widetilde{U_\epsilon}(x-q_{2j-1}), &  x \in [q_{2j-1}, q_{2j-1}+ \rho_\epsilon], \; q_{2j-1}=x_{2j-1}-\theta_{2j-1} \rho_\epsilon,
  \; \mbox{when} \; 2j-1 \leq k_\delta, \\
  0, & \mbox{if} \; x \leq -i_\delta -\rho_\epsilon, \\
\chi_E(x), & \mbox{otherwise}.
\end{array} \right.
\]
At a 
 point $x_{2j}$,
\begin{eqnarray*}
\mbox{if} \; \theta_{2j}=0, & \mbox{then} \;
\int_{x_{2j}-\rho_\epsilon}^{x_{2j}+\rho_\epsilon} e^x U_\epsilon^2 \, dx & \leq \int_{x_{2j}-\rho_\epsilon}^{x_{2j}+\rho_\epsilon} e^x \chi_E \, dx\,,\\
\mbox{if} \; \theta_{2j}=1, & \mbox{then} \;
\int_{x_{2j}-\rho_\epsilon}^{x_{2j}+\rho_\epsilon} e^x U_\epsilon^2 \, dx & \geq \int_{x_{2j}-\rho_\epsilon}^{x_{2j}+\rho_\epsilon} e^x \chi_E \, dx\,.
\end{eqnarray*}
Similar inequalities hold at $x_{2j-1}$.
By the immediate value theorem we can select each $\theta_i \in [0,1]$ to ensure that $\int_{-\infty}^{\infty} e^x w_\epsilon^2 \, dx=1$.
Observe that $w_\epsilon \in Y$ and $w_\epsilon(x) =0$ when $x \geq x_1+1$.

Since $\rho_\epsilon \to 0$ as $\epsilon \to 0$, it follows that $w_\epsilon \to \chi_E$ pointwise a.e. on $[-i_\delta-1, x_1+1]$.
The dominated convergence theorem then gives $w_\epsilon \to \chi_E$ in $L^2_e(-i_\delta, x_1+1)$. Now
$\int_{-\infty}^{-i_\delta} e^x |w_\epsilon - \chi_E|^2 \, dx \leq \int_{-\infty}^{-i_\delta} e^x dx \leq \delta$. Since $\delta$ is arbitrary,
we conclude that $w_\epsilon \to \chi_E$ in $L^2_e$.
Together with $0 \leq w_\epsilon \leq 1$, the stability theorem is applicable to the nonlocal term and the $G$-integral. 
Thus it suffices to check the limsup inequality for only the gradient term and the $F_0$-integral. 

\begin{eqnarray*}
 && \int_{-\infty}^\infty e^x ( \frac{\epsilon w'^{\,2}_\epsilon}{2} + \frac{F_0(w_\epsilon)}{\epsilon} ) \, dx \\
& = & \sum_{j=1}^{k_\delta} \int_{q_j}^{q_j +\rho_\epsilon} 
e^x ( \frac{\epsilon w'^{\,2}_\epsilon}{2} + \frac{F_0(w_\epsilon)}{\epsilon} ) \, dx \\
& =&  \sum_{j~\mbox{is even}}^{k_\delta} e^{q_j} \int_{0}^{ \rho_\epsilon}
e^x ( \frac{\epsilon U'^{\,2}_\epsilon}{2} + \frac{F_0(U_\epsilon)}{\epsilon} ) \, dx  
+ \sum_{j~\mbox{is odd}}^{k_\delta} e^{q_j} \int_{0}^{ \rho_\epsilon} 
e^x ( \frac{\epsilon \widetilde{U_\epsilon}'^{\,2}}{2} + \frac{F_0(\widetilde{U_\epsilon})}{\epsilon} ) \, dx \\
& \leq &  \sum_{j~\mbox{is even}}^{k_\delta} e^{q_j} e^{\rho_\epsilon} \int_{0}^{ \rho_\epsilon}
 ( \frac{\epsilon U'^{\,2}_\epsilon}{2} + \frac{F_0(U_\epsilon)}{\epsilon} ) \, dx  
+ \sum_{j~\mbox{is odd}}^{k_\delta} e^{q_j} e^{\rho_\epsilon} \int_{0}^{ \rho_\epsilon} 
 ( \frac{\epsilon \widetilde{U_\epsilon}'^{\,2}}{2} + \frac{F_0(\widetilde{U_\epsilon})}{\epsilon} ) \, dx \\
& =&  \sum_{j}^{k_\delta} e^{q_j} e^{\rho_\epsilon} \int_{0}^{ \rho_\epsilon}
( \frac{\epsilon U'^{\,2}_\epsilon}{2} + \frac{F_0(U_\epsilon)}{\epsilon} ) \, dx  
\quad \mbox{by using symmetry}, \\
& \leq &  \sum_{j}^{k_\delta} e^{q_j} e^{\rho_\epsilon} \int_{0}^{ \rho_\epsilon} 
 \;  \frac{\epsilon + 2 F_0(U_\epsilon)}{\epsilon}  \, dx   \quad \mbox{by using} \; (\ref{U1}), \\
 & = &  \sum_{j}^{k_\delta} e^{q_j} e^{\rho_\epsilon} \int_{0}^{ \rho_\epsilon}
 \sqrt{\epsilon + 2 F_0(U_\epsilon)} \;U'_\epsilon \, dx    \\
  & = &  \sum_{j}^{k_\delta} e^{q_j} e^{\rho_\epsilon} \int_{0}^{ 1}
 \sqrt{\epsilon + 2 F_0(s)} \, ds    \;.
\end{eqnarray*}
Taking limit as $\epsilon \to 0$ yields
\begin{eqnarray*} \label{sup1}
\lim \sup  \int_{-\infty}^\infty e^x (\epsilon \frac{w'^{\,2}_\epsilon}{2} + \frac{F_0(w_\epsilon)}{\epsilon} ) \, dx 
&\leq  &\phi(1) \sum_j^{k_\delta} e^{x_j}  \\
&\leq & {\frac{\sqrt{2}}{12}} \sum_j^\infty (e^{a_j}+e^{b_j}) \;,
\end{eqnarray*}
from which 
the proof 
{can be completed.}
\end{proof}

By virtue of Lemma~\ref{lem_liminf} and Lemma~\ref{lem_limsup}, 
the $\Gamma$-limit of $J_{c(\epsilon)}$ {in $L^2_e$} has been established;  
\begin{equation} \label{gamma_conv}
\mbox{$\Gamma$-$\lim$}\,J_{c(\epsilon)} = J^*_{c_0} \;.
\end{equation}
We claim that  $\liminf\limits_{\epsilon \to 0} ( \inf_{w \in Y} J_{c(\epsilon)}(w))< \infty$. Indeed taking
\[
w_\epsilon (x)= \left\{ \begin{array}{ll} 1, & \mbox{if} \; x\leq 0, \\
                                        1-\frac{x}{\epsilon}, & \mbox{if}  \;  0 \leq x \leq \epsilon, \\
                                         0, & \mbox{if} \; x \geq \epsilon,  
                                         \end{array} \right.
\]   
 we see that $w_\epsilon(\cdot-\theta_\epsilon) \in Y$ for some $\theta_\epsilon \to 0$. A direct calculation yields      
$\liminf\limits_{\epsilon \to 0}  J_{c(\epsilon)}(w_\epsilon)< \infty$, which verifies the claim. 

Since there exists  
 $C_4>0$ 
 such that $G(\xi) \geq -C_4 \xi^2$ for $\xi \in [-\widetilde{M_1}-1,\beta_2]$,
if $w \in Y$ then 
\begin{eqnarray*}
J_{c(\epsilon)}(w) & \geq  & \alpha \int_{-\infty}^\infty e^x G(w) \, dx \\
& \geq & - \alpha C_4 \int_{-\infty}^\infty e^x w^2 \, dx \\
& =& - \alpha C_4  \;.
\end{eqnarray*}
Hence \[
{\infty>}  \liminf\limits_{\epsilon \to 0} ( \inf_{w \in Y} J_{c(\epsilon)}(w)) \geq -\alpha C_4 \;
\]
%
and there is 
 a sequence
$\{\zeta_{\epsilon}\} \subset Y$ such that 
\begin{equation} \label{U}
\lim_{\epsilon \to 0} J_{c(\epsilon)}(\zeta_{\epsilon}) =  \liminf\limits_{\epsilon \to 0} ( \inf_{w \in Y} J_{c(\epsilon)}(w)) \;. 
\end{equation}
Such a minimizing sequence will be utilized to establish the existence of minimizers for $J^*_{c_0}$ in the next section. 







\section{
{Minimizer for $J_c^*$}} \label{min}
\setcounter{equation}{0}
\renewcommand{\theequation}{\thesection.\arabic{equation}}

In the investigation of the minimizers of $J^*_{c_0}$, {we assume} $c(\epsilon) \to c$ as $\epsilon \to 0$ throughout the section; 
{it allows us to use} a simpler notation $J^*_{c}$ instead of $J^*_{c_0}$. 
Also, if $\chi_E \in BV_e$ we will follow the formula stated in (\ref{Jinfty}) to evaluate $J^*_{c}(\chi_E)$
even if $\| \chi_E \|_{L^1_e} \ne 1$.

\begin{lemma} \label{lem_E1} 
If $c(\epsilon) \to c$ as $\epsilon \to 0$, then 
$J_c^*$ has a minimizer, denoted by ${\chi_{E_c}}$. Moreover \\ 
(i) $\int_{-\infty}^\infty e^x 
\chi_{E_c} \,dx=1$;\\
(ii) 
$E_{c}=[a,b]$ for some $a<b$; the scenario $a=-\infty$ is allowed.
\end{lemma}
\begin{proof}
 Since $\liminf\limits_{\epsilon \to 0} ( \inf_{w \in Y} J_{c(\epsilon)}(w)) \leq C_0$ {for some $C_0>0$}, it follows from (\ref{U})
and Lemma~\ref{lem_compact} 
 that $\zeta_{\epsilon} \to \chi_{E_{c}}$
in $L^2_e$ with $\| D \chi_{E_{c}} \|_e \leq 6 \sqrt{2} (C_0+C_1 \alpha)$. By a fundamental theorem of $\Gamma$-convergence
\cite[Theorem 2.1]{Braides2},
$\chi_{E_{c}}$ is a minimizer of $J_c^*$ and 
\begin{equation} \label{key}
J_c^*(\chi_{E_{c}})=\lim_{\epsilon \to 0} J_{c(\epsilon)}(\zeta_{\epsilon}). 
\end{equation} 
Moreover $\| \chi_{E_{c}} \|_{L^1_e}= \| \chi_{E_{c}} \|_{L^2_e}= \lim_{\epsilon \to 0} \| \zeta_{\epsilon} \|_{L^2_e}=1$. {Thus} 
(i) is proved. \\

Suppose 
$E_{c}= \cup_{i=1}^\infty [a_i,b_i]$ with $b_1>a_1>b_2>a_2> \dots$. 
Let $E^1=[a_1,b_1]$ and $E^2=E_c \setminus E^1$ with both $E^1$ and $E^2$ being non-empty.  
For $i=1,2$, let $\| \chi_{E^i} \|_{L^1_e}=e^{-h_i}$ for some $h_i>0$. Thus 
\begin{equation} \label{sum}
1=\| \chi_{E_{c}} \|_{L^1_e}= e^{-h_1}+e^{-h_2} 
\end{equation}
and  $\| \chi_{E^i}(\cdot-h_i) \|_{L^1_e}=1$. 
Since $\chi_{E_{c}}$ is a minimizer of $J_c^*$, it is clear that $J_c^*(\chi_{E_{c}}) \leq 
J_c^*(\chi_{E^i}(\cdot-h_i))=  e^{h_i} J_c^*(\chi_{E^i})$.
A simple calculation gives
\begin{equation} \label{E12}
J_c^*(\chi_{E_{c}})=J_c^*(\chi_{E^1}) + J_c^*(\chi_{E^2}) + \sigma_0  \int_{-\infty}^{\infty} e^x \chi_{E^1} {\cal L}_{c} \chi_{E^2} \, dx \;.
\end{equation}
Then 
\[
J_c^*(\chi_{E_{c}})  \geq  e^{-h_1} J_c^*(\chi_{E_{c}})+
e^{-h_2} J_c^*(\chi_{E_{c}})
+ \sigma_0  \int_{-\infty}^{\infty} e^x \chi_{E^1} {\cal L}_{c} \chi_{E^2} \, dx \;,
\]
 together with (\ref{sum}) leads to 
\[
0 \geq \sigma_0  \int_{-\infty}^{\infty} e^x \chi_{E^1} {\cal L}_{c} \chi_{E^2} \, dx  \;.
\]
This is absurd, since ${\cal L}_c \chi_{E^2}\, (x) >0$ for all $x$, which implies $\int_{-\infty}^{\infty} e^x \chi_{E^1} {\cal L}_{c} \chi_{E^2}dx > 0$. 
{As a conclusion, $E_{c}=[a,b]$ or $(-\infty,0]$; in the latter case 
$\| \chi_{E_c} \|_{L^1_e}=1$.}
\end{proof}



{Note that $E_c$ can be a finite or semi-infinite interval. We need preliminary analysis in oder to distinguish such two cases in later sections.} 
Define $\ell=b-a$. With $\| \chi_{[a,b]} \|_{L^1_e}=e^{b}-e^{a}=1$, a {direct} calculation from (\ref{Jinfty}) gives
\begin{eqnarray} 
e^{-b} J^*_{c}(\chi_{[a,b]}) 
&=&J^*_{c}(\chi_{[-\ell,0]}) \nonumber \\
&=&\frac{\sqrt{2}}{12} (e^{-\ell}+1) - \frac{ \sqrt{2} \, \alpha}{12} (1-e^{-\ell}) +\frac{\sigma}{2}
 \int_{-\ell}^0 e^x {\cal L}_c \chi_{[-\ell,0]} \, dx  \,. 
 \label{JtE}
\end{eqnarray}
We next calculate the nonlocal term. 
The solutions of the characteristic 
equation $c^2 r^2+c^2 r-\gamma=0$ are
\begin{equation} \label{r21}
r = \frac{1}{2c}(-c \pm \sqrt{c^2+4 \gamma} ) \;,
\end{equation}
which will be denoted by $r_1$, $r_2$; here 
 $r_1<-1<0<r_2$ and note that $r_1+r_2=-1$.
The general solution of 
 $c^2 v''+c^2 v'- \gamma v=0$ is an element of 
$span \{e^{r_1 x}, e^{r_2 x}\}$. 
Solving \eqref{FNTW2} with $u=\chi_{[-\ell,0]}$ and $\ell$ being finite, we obtain 
\begin{equation} \label{vinfty}
{\cal L}_c \chi_{[-\ell,0]}=  \left\{ \begin{array}{ll}
A_4 e^{r_2 x}, & \mbox{for} \; x \leq -\ell, \\
\frac{1}{\gamma}+ A_3 e^{r_1x} + A_2 e^{r_2 x}, & \mbox{for} \; - \ell \leq x \leq 0, \\
A_1 e^{r_1 x}, & \mbox{for} \; x \geq 0,
\end{array}  \right.  
\end{equation}
where 
\[
A_1=\frac{r_2(1-e^{r_1 \ell})}{\gamma (r_2-r_1)}, \quad
A_2=\frac{r_1}{\gamma (r_2-r_1)}\;, \quad A_3= \frac{r_2 e^{r_1 \ell}}{\gamma (r_1-r_2)} \;, \quad
A_4=\frac{r_1(1-e^{r_2 \ell})}{\gamma (r_2-r_1)} \;.
\]
Evaluating the nonlocal term in (\ref{JtE}), we get 
\begin{eqnarray*}
\int_{-\ell}^0 e^x {\cal L}_c \chi_{[-\ell,0]} \, dx 
&=& \int_{-\ell}^0 e^x ( \frac{1}{\gamma} + A_3 e^{r_1 x} + A_2 e^{r_2 x}) \, dx \\
& =& \frac{2}{\gamma (r_2-r_1)} (r_2+r_1 e^{-\ell}+ e^{r_1 \ell}) \;.
\end{eqnarray*}
Substituting into (\ref{JtE}) gives 
\[
{ J_c^*(\chi_{[a,b]})= e^{b} {\cal J}(\ell, c) } \;
 \]
if we define 
\begin{eqnarray} \label{JtE2}
{{\cal J}(\ell, c)} & \equiv & \frac{\sqrt{2}}{12} (1-\alpha) + \frac{\sqrt{2}}{12} (1+\alpha)e^{-\ell} +
 \frac{\sigma}{\gamma (r_2-r_1)} (r_2+r_1 e^{-\ell}+ e^{r_1 \ell}) \;.
 \end{eqnarray}
 {As a reminder, the last term on the right hand side of \eqref{JtE2} is positive, since
  it is generated from $\int_{-\ell}^0 e^x {\cal L}_c \chi_{[-\ell,0]} \, dx$.} \\
 
 Recall that $E_{c}=[a,b]$ and $1=\| \chi_{[a,b]} \|_{L^1_e}=e^b (1-e^{-\ell})$. Hence 
 $b= - \log(1-e^{-\ell})$ and 
  \begin{equation} \label{JtE1b}
 J^*_{c}(\chi_{E_{c}})= \frac{1}{1-e^{-\ell}} \,
 {\cal J}(\ell, c).  \;
 \end{equation}
 It has 
 been shown that {$E_c$ is a minimizer of
 $J_c^*$.
 To 
 see} how 
  such a minimizer depends on $c$, we introduce an auxiliary function $H(c)= \frac{c}{\sqrt{c^2+4 \gamma}}$. Clearly $H$ is strictly increasing, 
  $H(0)=0$ and $H(c) \to 1$ as $c \to \infty$. 
  Straightforward calculation gives
\begin{eqnarray} 
\ds \frac{r_2}{r_2-r_1}= \frac{1}{2} (1-H(c)), && \frac{r_1}{r_2-r_1}= - \frac{1}{2} (1+H(c)), \label{H1} \\
\ds  \frac{d r_1}{dc} >0,  && \frac{d r_2}{dc} <0,  \label{H1bb} \\
\ds r_2-r_1=\frac{1}{H(c)}, &&  \frac{1}{(r_2-r_1)^2} \frac{d (r_2-r_1)}{dc}= -H'(c).
\label{H1aa}
\end{eqnarray}
For 
a given $c$, we intend to solve a 
 $\ell_c \in (0,\infty]$ such that
 \begin{equation} \label{lc}
  \frac{1}{1-e^{-\ell_c}} \,{\cal J}(\ell_c, c) \leq  \frac{1}{1-e^{-\ell}} \,
 {\cal J}(\ell, c)  \quad \mbox{holds for all} \; \ell>0.
 \end{equation}
 {When $c$ is the correct wave speed, the case $\ell_c=\infty$ corresponds to a traveling front for the limiting problem, while a finite $\ell_c$ indicates a
   traveling pulse.}

\section{Traveling front} \label{fr}
\setcounter{equation}{0}
\renewcommand{\theequation}{\thesection.\arabic{equation}}

In {this} section we examine the 
 case 
  $\ell_c=\infty$. 
  Define 
 \begin{eqnarray} \label{calF} 
 {\cal F}(c) \equiv  \lim_{\ell \to \infty} {{\cal J}(\ell,c)} 
  = \frac{\sqrt{2}}{12} (1-\alpha) + \frac{\sigma}{2 \gamma}(1-H(c)). \; 
\end{eqnarray}
It is clear that 
\begin{eqnarray} \label{lem_decr}
{\cal F}'(c)=- \frac{\sigma}{2 \gamma} H'(c)<0.
\end{eqnarray}
\begin{lemma} \label{lem_large_l}
For 
 $c>0$, {a necessary and sufficient condition} for $\ell_c=\infty$ is 
\begin{equation} \label{cond_infty}
{\frac{\sqrt{2} \, \gamma}{6 \sigma} {\geq} H(c) \;.}
\end{equation}
Under the hypothesis \eqref{cond_infty}, $\chi_{(-\infty,0]}$ is the unique minimizer of $J_c^*$.
\end{lemma}
\begin{proof}
A direct calculation gives
\begin{equation}
 \frac{1}{1-e^{-\ell}} {\cal J}(\ell, c) - {\cal F}(c) 
= \frac{1}{1-e^{-\ell}} \{ \frac{\sqrt{2}}{6} e^{-\ell} - \frac{\sigma H(c)}{\gamma} e^{-\ell}
+ \frac{\sigma}{\gamma (r_2-r_1)} e^{r_1 \ell} \} \;. \label{difference}
\end{equation}
{If $\ell_c=\infty$ then $ \frac{1}{1-e^{-\ell}} {\cal J}(\ell, c) - {\cal F}(c) \geq 0$ for 
 all $\ell >0$. In particular for large $\ell$, it follows from (\ref{difference}) and $r_1<-1$ 
 that (\ref{cond_infty}) holds.} \\

{
As to show that (\ref{cond_infty}) is a sufficient condition, it is clear from (\ref{difference}) that
 $\frac{1}{1-e^{-\ell}} {\cal J}(\ell, c) - {\cal F}(c) > 0$ for all $\ell$. Since 
 the last inequality is strict, 
 $\chi_{(-\infty,0]}$ is the unique minimizer of $J_c^*$.}
\end{proof}

  



 
{Besides $\ell_c=\infty$, we next investigate the additional constraints on the physical parameters such that ${\cal F}(c)=0$ is satisfied as well.} We now aim at the following assumption: \\

\noindent
{$(A 1)^*$} $\alpha \geq \frac{3 \sqrt{2} \,\sigma}{ \gamma} > \alpha-1 > 0$.
\begin{lemma} \label{lem_front}
 {$(A 1)^*$} is a 
necessary and sufficient condition 
 for both (i) and (ii) to be held: \\ 
(i) there is a unique {$c_f>0$}      
such that $
{\cal F}(c_f)=0$. \\ 
(ii) 
$\chi_{(-\infty,0]}$ is the unique global minimizer of $J^*_{c_f}$ and 
$J^*_{c_f}(\chi_{(-\infty,0]})=0$.\\
\end{lemma}
\begin{proof}
{First we prove the sufficiency.}
If $\alpha>1$ and 
 $\frac{3 \sqrt{2}\, \sigma}{\gamma} > (\alpha-1)$, then by direct calculation 
$
{\cal F}(0)=\frac{\sqrt{2}}{12} (1-\alpha)+ \frac{\sigma}{2 \gamma} > 0$
and $
{\cal F}(+\infty)=\frac{\sqrt{2}}{12} (1-\alpha)<0$. 
Since ${\cal F}
$ is strictly decreasing in $c$, 
 there is
 a unique $c_f$
such that $
{\cal F}(c_f)=0$. 
This together with (\ref{calF}) yields
\begin{equation} \label{cinfty}
H(c_f)=1 - \frac{2 \phi(1) (\alpha-1) \gamma}{\sigma}. 
\end{equation}
Recall that $h_*=1 - \frac{ (\alpha-1) \gamma}{3 \sqrt{2} \, \sigma}$. Solving (\ref{cinfty}) yields \eqref{cinfty1}. 
Moreover, to ensure $\ell_c=\infty$ at the same time, we need to check that (\ref{cond_infty}) holds with $c=c_f$. This is true if 
$\alpha {\geq} \frac{3 \sqrt{2}\, \sigma}{\gamma}$. \\

Conversely if (i) holds, this unique $c_f$ has to be determined by \eqref{cinfty1}. Since ${\cal F}(c_f)=0$ and ${\cal F}$ is strictly decreasing, 
it follows that ${\cal F}(0)>0$ and ${\cal F}(\infty)<0$. Together with \eqref{cond_infty} lead to {$(A 1)^*$}. 
\end{proof}
\begin{remark} \label{remark_c_f}
(a) As a consequence of \eqref{cinfty1}, $c_f$ is an increasing function of $h_*$. Thus $c_f$ is a decreasing function of $\alpha$ but an increasing function of $\sigma$. 
{In fact $c_f$ can get close to $0$ or tend to $\infty$ with appropriate values on the parameters.} \\ %
(b) Note that \eqref{cinfty1} gives a formula to calculate 
 the $\Gamma$-limit 
  speed for the traveling front solutions. 
\end{remark}

{In Theorem~\ref{thm_front}, a stronger condition ($A 1$) is imposed; that is, requiring {$(A 1)^*$} be satisfied with strict inequalities. It enables us to prove the existence of 
traveling wave solutions of \eqref{FNTW1}-\eqref{FNTW2} when $\epsilon$ is small, as stated in Theorem~\ref{thm_front}.} \\

%



\noindent
{\em Proof of Theorem~\ref{thm_front}.}
{Since $c_f$ satisfies \eqref{cond_infty} with a strict inequality}, 
there is an $\eta>0$ 
such that if $c \in [c_f-\eta,c_f+\eta]$ then (\ref{cond_infty}) {continue to} hold  and thus 
 $\inf_{L^2_e} J_{c}^*= 
{\cal F}(c)$.
Furthermore, ${\cal F}$ is a strictly decreasing function of $c$, hence  
${\cal F}(c_f-\eta)>0={\cal F}(c_f)>{\cal F}(c_f+\eta)$. \\

Let $c^+=c_f+\eta$ and $c^-=c_f-\eta$. For $w \in Y$, define 
\begin{eqnarray}
I_{\epsilon,c}(w)
&=&  \int_{-\infty}^\infty e^x \{ \epsilon \frac{w'^{\,2}}{2} + \frac{F_0(w)}{\epsilon} + {\alpha} G(w) 
+ \frac{\sigma}{2} w {\cal L}_{c} w \}\,dx . \label{IJ1}
\end{eqnarray}
{In other words, we take $I_{\epsilon,c}=J_{c(\epsilon)}$ with $c(\epsilon)=c$ for all $\epsilon$. Consequently}
\begin{equation} \label{gamma_conv_I}
\mbox{$\Gamma$-$\lim$}\,{I_{\epsilon,c}} = J^*_{c} \;.
\end{equation}
Hence there is an $\epsilon_0 > 0$ such that 
 $\inf_Y I_{\epsilon,c^-}> 0$ and $\inf_Y I_{\epsilon,c^+}<0$ if $\epsilon < \epsilon_0$. With only slight modification, the argument of 
  Lemma 3.5 of \cite{CC1} shows that $\inf_Y {I_{\epsilon,c}}$ is a continuous function of $c$. 
  By the {intermediate} value
theorem, 
 there exists
 a $c_\epsilon \in (c_f-\eta,c_f+\eta)$ such that $\inf_Y I_{\epsilon,c_\epsilon}= 0$.
 Furthermore by Lemma~\ref{lem_minimum} and Lemma~\ref{lem_soln}, $I_{\epsilon,c_\epsilon}$ has a minimizer $u_\epsilon \in Y$. 
 {Together with $I_{\epsilon, c_\epsilon}(u_\epsilon)=0$, }
we set $v_\epsilon={\cal L}_{c_\epsilon}u_\epsilon$ to obtain 
 a traveling wave solution $(u_\epsilon,v_\epsilon)$ of \eqref{FNTW1}-\eqref{FNTW2} with a speed $c_\epsilon$. With $c_\epsilon \in (c_f-\eta,c_f+\eta)$
 and Lemma~\ref{lem_compact}, as $\epsilon \to 0$ we see that along a subsequence  $c_\epsilon\ \to \hat{c_f}$ 
 and $u_\epsilon \to \chi_E$ in $L^2_e$ for some $\hat{c_f} \in {\mathbb R}$ and $\chi_E \in BV_e$. {As a consequence of $\Gamma$-convergence, $\chi_{E}$ is a minimizer of $J^*_{\hat{c_f}}$ and
\[
0=\lim_\epsilon  I_{\epsilon,c_\epsilon}(u_\epsilon)=J^*_{\hat{c_f}}(\chi_E) \;.
\]
By Lemma~\ref{lem_front}, $\hat{c_f}=c_f$ and $\chi_E=\chi_{(-\infty,0]}$. {Finally it follows from the uniqueness that  the limit of each subsequence is the same. Hence the convergence $c_\epsilon \to c_f$ and $u_\epsilon \to \chi_{(-\infty,0]}$ is along the whole sequence.}
\hfill $\Box$}


\vspace{.1in}
To give a physical interpretation for {the} assumption {($A 1$)}, we take a look of the nullclines.
Denoted by $\gamma_*=\gamma_*(\epsilon)$ such that
 the two regions enclosed by the  line $v=u/\gamma_*$ and 
the curve $v={f_\epsilon}(u)/{\sigma}$ 
are equal in area with opposing signs.

\begin{lemma} \label{lem_gstar}
If $\epsilon$ is sufficiently small then $\gamma_*= \frac{3 \sqrt{2} \, \sigma}{\alpha}(1+O(\epsilon))$.
\end{lemma}
\begin{proof}
The nullclines $v=u/\gamma_*$ and $v={f_\epsilon}(u)/{\sigma}$ 
intersect
at $u=0, \mu_2^*, \mu_3^*$, where $\mu_2^*, \mu_3^*$ are the roots of the quadratic equation
$u^2-(1+\beta_\epsilon)u + \beta_\epsilon+ \frac{\epsilon \sigma}{\gamma_*}=0$. As both nullclines are anti-symmetric about 
$(u,v)=(\mu_2^*,\mu_2^*/\gamma_*)$, it can be easily checked $\mu_2^*=\frac{1+\beta_\epsilon}{3}$ 
and $\mu_3^*=\frac{2(1+\beta_\epsilon)}{3}$.
Using the fact that $\mu_2^* \mu_3^*=\beta_\epsilon+ \frac{\epsilon \sigma}{\gamma_*}$, we obtain
\begin{equation} \label{gam}
\frac{\epsilon \sigma}{\gamma_*}= \frac{1}{9}(1-2\beta_\epsilon) (2-\beta_\epsilon) \;.
\end{equation}
The lemma follows by substituting 
 $\beta_\epsilon=\frac{1}{2} - \frac{\epsilon \alpha}{\sqrt{2}}$ into \eqref{gam}.
\end{proof}
For sufficiently small $\epsilon$, Lemma~\ref{lem_gstar} shows 
 that {($A 1$)} 
 is equivalent to
\begin{equation} \label{phy2}
\alpha>1 \quad \mbox{and} \quad \frac{\alpha}{\alpha-1} \gamma_* >\gamma >\gamma_* \;,
\end{equation}
which can be viewed as the constraints on $\gamma$ to generate a wave front. 


\section{Traveling pulse} \label{pul}
\setcounter{equation}{0}
\renewcommand{\theequation}{\thesection.\arabic{equation}}

In this section our attention turns to 
 traveling pulse. Let $c > 0$ and suppose that there is an 
 $\ell_c \in (0,\infty)$ such that
 \begin{equation} \label{lcp}
  \frac{1}{1-e^{-\ell_c}} \,{\cal J}(\ell_c, c) \leq  \frac{1}{1-e^{-\ell}} \,
 {\cal J}(\ell, c)  \quad \mbox{holds for all} \; \ell>0 \ .
 \end{equation}
 It follows that 
  \begin{equation} \label{Jp'=0}
  \frac{\partial}{\partial \ell}
  \left( \frac{{\cal J}(\ell,c)}{1-e^{-\ell}} \right) \Big|_{\ell=\ell_c} =0 \ .
  \end{equation}
For the existence of traveling pulses, in addition to \eqref{Jp'=0} we look for suitable ranges for the parameters under which a value of $c$ can be found to satisfy 
\begin{equation} \label{J=0}
{\cal J}(\ell_c,c)=0 \;.
\end{equation} 
 Then $\chi_{[a,b]}$ is a 
 minimizer of $J_c^*$ if we take $b= - \log (1-e^{-\ell_c})$ and $a=b-\ell_c$. A lower bound of $\ell_c$ is given in the next lemma. 

\begin{lemma} \label{lem_width}
If  $\chi_{[a,b]}$ is a 
 minimizer of $J_c^*$ {satisfying $J_c^*(\chi_{[a,b]})=0$} then {$\alpha>1$ and} 
 $\ell_c=b-a
  > \log  \frac{\alpha+1}{\alpha-1}$. 
 \end{lemma}

\begin{proof}
 It simply follows from  
 ${\cal J}(\ell_c, c)=0$ and the last term in \eqref{JtE2} is positive. 
\end{proof}

\begin{lemma} \label{lem_pulse1}
Let $c>0$ and assume that $J_{c}^*$ has a minimizer $\chi_{[a,b]}$ with $\| \chi_{[a,b]} \|_{L^1_e}=1$ and $J_{c}^*(\chi_{[a,b]})=0$. If $\ell_c=b-a$ then 
 $(\ell_c,c)$ satisfies 
\begin{eqnarray} 
\frac{\sigma}{2 \gamma} (1+H(c)) \; (1-e^{-r_2 {\ell_c}}) &= & \frac{\sqrt{2}}{12} (\alpha +1) \;  \label{H2b} 
\end{eqnarray}
and 
\begin{eqnarray} 
\frac{\sigma}{2 \gamma} (1-H(c)) \; (1-e^{r_1 \ell_c}) &= & \frac{\sqrt{2}}{12} (\alpha -1) \;. \label{H2a}
\end{eqnarray}
\end{lemma}
\begin{proof}
By direct calculation 
\begin{eqnarray} 
\frac{\partial }{\partial \ell} {\cal J} (\ell, c) 
&= & - \frac{\sqrt{2}}{12} (1+\alpha)  e^{-\ell} - (1+H(c))
\frac{\sigma }{2\gamma } (-  e^{-\ell} + e^{r_1 \ell}) \;.\label{Della}
\end{eqnarray}
 Since both \eqref{Jp'=0} and \eqref{J=0} must hold, 
  we obtain 
  \begin{equation} \label{Jp=0}
  \frac{\partial {\cal J}}{\partial \ell} (\ell_c,c)=0 \;.
  \end{equation} 
This gives \eqref{H2b}. Using it together with 
 \eqref{J=0} yields
 \[
 \frac{\sqrt{2}}{12} (1-\alpha) + \frac{\sigma}{\gamma(r_2-r_1)} (r_2+e^{r_1 \ell}) 
 +\frac{\sigma r_1}{\gamma(r_2-r_1)} e^{r_1 \ell}=0 \ ,
 \]
from which 
 \eqref{H2a} follows. 
\end{proof}

\begin{lemma} \label{lem_pulse2}
Let 
$\sigma$ and $\gamma$ be given. If $J_{c}^*$ has a minimizer $\chi_{[a,b]}$ with $\| \chi_{[a,b]} \|_{L^1_e}=1$ and $J_{c}^*(\chi_{[a,b]})=0$, then 
\begin{equation}
\frac{3 \sqrt{2} \, \sigma}{\gamma} > \alpha >1 \ . \label{pulse_para}
\end{equation}
\end{lemma}
\begin{proof}
By Lemma~\ref{lem_pulse1}, $(\ell_c,c)$ satisfies \eqref{H2b} and \eqref{H2a}. 
Clearly the left hand side of \eqref{H2a} is positive,
which implies $\alpha>1$; {a fact which is already indicated in Lemma~\ref{lem_width}.}
Define 
\begin{equation} \label{H3}
Q(\ell, c) \equiv \frac{\sqrt{2} \, \gamma }{12 \sigma} \{ \frac{1+\alpha}{1-e^{-r_2 \ell}} + \frac{\alpha-1}{1-e^{r_1 \ell}} \} -1 \;.
\end{equation}
From \eqref{H2b} and \eqref{H2a}, eliminating
$H(c)$ yields $Q(\ell_c, c) = 0.$ \\ 

By direct calculation 
 \begin{equation} \label{dQdl}
 \frac{\partial Q}{\partial \ell} = \frac{\sqrt{2} \, \gamma}{12 \sigma} \{ - \frac{(1+\alpha)r_2}{(1-e^{-r_2 \ell})^2} \, e^{-r_2 \ell} \,
+ \frac{(\alpha-1) r_1}{(1-e^{r_1 \ell})^2} \, e^{r_1 \ell}  \}< 0 \; .
 \end{equation}
Since $\frac{\partial Q}{\partial \ell}<0$ for all $\ell$ and $Q(\ell_c, c) = 0$, 
 it follows that $0 > \lim_{\ell \to \infty}Q(\ell,c)=  \frac{\sqrt{2} \, \alpha \gamma }{ 6 \sigma} -1$, which 
 completes the proof. 
\end{proof}

{
\begin{remark} \label{remark_Q}
If we get a solution from 
 ${\cal J}(\ell,c)= \frac{\partial {\cal J}}{\partial \ell} (\ell,c)=0$, it satisfies \eqref{H2b}-\eqref{H2a}, and is also a solution of 
  $\frac{\partial {\cal J}}{\partial \ell} (\ell,c)=Q(\ell,c)=0$. 
As $Q$ is a linear combination of ${\cal J}$ and $\frac{\partial {\cal J}}{\partial \ell}$, we expect that
this 
 solution satisfies $Q(\ell, c)={\cal J}(\ell,c)=0$.
\end{remark}
}

 \begin{lemma} \label{lem_pulse3}
If ${\cal J} (\ell_c,c)=0$ and $\frac{\partial {\cal J}}{\partial \ell} (\ell_c,c)=0$, then 
$\frac{\partial^2 {\cal J}}{\partial \ell^2} (\ell_c,c)>0$.
\end{lemma}
\begin{proof}
Direction calculation gives 
\begin{equation} \label{Della2}
\frac{\partial^2 }{\partial \ell^2} {\cal J} (\ell, c)=   \frac{\sqrt{2}}{12}  e^{-\ell} [(1+\alpha)- (1+H(c))
\frac{3 \sqrt{2} \, \sigma }{\gamma} (1 + r_1e^{-r_2 \ell})] \;.
\end{equation}
Invoking (\ref{H2b}) 
 yields
\begin{equation} \label{Della2b}
\frac{\partial^2 {\cal J}}{\partial \ell^2} (\ell_c, c)=  - \frac{\sqrt{2}}{12}  e^{-\ell_c} \; [ 
\frac{3 \sqrt{2} \, \sigma }{\gamma } (1+H(c)) (1 + r_1)e^{-r_2 \ell_c}] 
 > 0\;.
\end{equation}
\end{proof}

{Lemma~\ref{lem_pulse2} shows that \eqref{pulse_para} is a necessary condition for the existence of the specified minimizer, we now intend to establish that
it is a sufficient condition as well.
 For given $\sigma$ and $\gamma$, if $\alpha \in (1,\frac{3 \sqrt{2} \, \sigma}{\gamma})$,
it follows from straightforward calculation that 
 \begin{equation} \label{dQdc}
\frac{\partial Q}{\partial c}= \frac{\sqrt{2} \, \gamma}{12 \sigma} \{ - \frac{(1+\alpha) \ell  }{(1-e^{-r_2 \ell})^2} \, e^{-r_2 \ell} \, \frac{d r_2}{dc}
+ \frac{(\alpha-1) \ell}{(1-e^{r_1 \ell})^2} \, e^{r_1 \ell} \, \frac{d r_1}{dc} \} >0 \; .
\end{equation}
For $c>0$, it is clear that $\frac{\partial Q}{\partial \ell}<0$ and 
 $\lim_{\ell \to 0^+}Q(\ell,c)=\infty$. Moreover by \eqref{pulse_para}, $\lim_{\ell \to \infty}Q(\ell,c)=  \frac{\sqrt{2}\, \alpha \gamma }{6 \sigma} -1 < 0$. 
Hence there is a unique 
 $L=L(c)$ satisfying $Q(L(c),c)=0$. Moreover 
 \begin{equation} \label{L'}
L'(c)= - \frac{ {\partial Q}/\partial c}{\partial Q/ \partial \ell}>0. \\
\end{equation}

Next we derive asymptotic properties of $L(c)$. As $c \to 0^+$,
it can be easily checked that $r_1 \sim -\frac{\sqrt{\gamma}}{c}(1+O(c))$ and $r_2 \sim \frac{\sqrt{\gamma}}{c}(1+O(c))$.
 Plugging 
  into 
  {$Q(L(c),c)=0$} yields
 \begin{equation} \label{asympL1}
\frac{\sqrt{2} \,\gamma}{12 \sigma}   \left( \frac{1+\alpha}{1-e^{-\sqrt{\gamma} L/c}} + \frac{\alpha-1}{1-e^{-\sqrt{\gamma} L/c}} \right)
= 1 +o(1)\;,
\end{equation}
which can be simplified as
 \begin{equation} \label{asympL2}
1- e^{-\sqrt{\gamma} L/c} = \frac{\sqrt{2}\, \gamma  \alpha}{6 \sigma} (1+o(1)).
\end{equation}
This shows {$\frac{L}{c} \to k_1 \equiv  - \frac{1}{\sqrt{\gamma}} \log ( 1- \frac{\sqrt{2}\, \gamma \alpha}{6 \sigma})$}
 as $c \to 0^+$. Then from \eqref{JtE2},
\[
\lim_{c \to 0^+} {\cal J}(L(c),c) = \frac{\sqrt{2}}{6} \;.
\]

Next if $c \to \infty$ then 
$r_2 \sim \frac{\gamma}{c^2}$ and $r_1 \sim -1$. {As $L(c)$ is an increasing function,}
 $\lim_{c \to \infty}L(c)$ exists;
  however it cannot be a finite number, 
from the fact of $Q(L(c),c)=0$ which then 
 gives
\begin{equation} \label{RL}
 \frac{1+\alpha}{1 - \lim_{c \to \infty} e^{- r_2 L}} + {\frac{\alpha -1}{1- \lim_{c \to \infty} e^{r_1L}}}  =\frac{6 \sqrt{2} \, \sigma}{\gamma}  \;.
\end{equation}
This indeed also 
eliminates the scenario that $\lim_{c \to \infty} r_2 L=0$. Furthermore, 
 $\limsup \limits_{c \to \infty} r_2 L =\infty$ would lead to 
$\frac{3 \sqrt{2} \, \sigma}{ \gamma} =\alpha$, which violates \eqref{pulse_para}. 
Thus the only possibility left is 
$\limsup \limits_{c \to \infty} r_2 L \equiv R_L$, a 
positive constant that can be solved from
\[
 \frac{1+\alpha}{1 - \lim_{c \to \infty} e^{- r_2 L}} + {\alpha -1} =\frac{6 \sqrt{2} \, \sigma}{\gamma}  \;.
\]
Consequently {along this subsequence},
\begin{equation} \label{asympL3}
{L \sim \frac{R_L c^2}{\gamma}}~\mbox{as}~c \to \infty
\end{equation}
 and then 
\[
\lim_{c \to \infty} {\cal J}(L(c),c)= \frac{\sqrt{2}}{12} (1-\alpha)  < 0 \ .
\]
By the intermediate value theorem there is a $c > 0$ such that 
${\cal J}(L(c),c)=0$. 
As indicated in Remark~\ref{remark_Q}, if 
 $(\ell,c)$ satisfies ${\cal J}(\ell,c)=Q(\ell,c)=0$, 
{then $\frac{\partial {\cal J}}{\partial \ell}(\ell,c)=0$ should be true as well. Indeed,} 

\begin{eqnarray} 
 \frac{\partial {\cal J}}{\partial \ell} (L(c),c)
&= &  - \frac{\sqrt{2}}{12} (1+\alpha)  e^{-L} 
+\frac{\sigma r_1}{\gamma (r_2-r_1)} (-  e^{-L} + e^{r_1 L}) \label{Dell1} \\
&= &  - \frac{\sqrt{2}}{12} (1+\alpha)  e^{-L} +\frac{\sigma r_1}{\gamma (r_2-r_1)} (-  e^{-L} + e^{r_1 L}) + {\cal J}(L(c),c) \;  \nonumber \\
&= &  \frac{\sqrt{2}}{12} (1-\alpha) + \frac{\sigma}{\gamma(r_2-r_1)} (r_2+e^{r_1 L}) 
 +\frac{\sigma r_1}{\gamma(r_2-r_1)} e^{r_1 L} \; \nonumber \\
&= &  \frac{\sigma}{2 \gamma} (1-H(c)) \; (1-e^{r_1 L}) - \frac{\sqrt{2}}{12} (\alpha -1) \;. \label{Della3} 
\end{eqnarray}
Also, 
we see from \eqref{Dell1} that 
\begin{eqnarray} 
e^{L} \frac{\partial {\cal J}}{\partial \ell} (L(c),c)
&= & \frac{\sigma}{2 \gamma} (1+H(c)) \; (1-e^{-r_2 L}) - \frac{\sqrt{2}}{12} (\alpha +1)  \;. \label{Della1} 
\end{eqnarray}
{Eliminating $H(c)$ from} \eqref{Della3}-\eqref{Della1} yields 
 \begin{equation} 
\frac{\partial {\cal J}}{\partial \ell} (L(c),c) = -\frac{\sigma}{\gamma}[(1-e^{r_1 L})^{-1} + e^{L}(1-e^{-r_2 L})^{-1}]^{-1} Q(L(c), c) = 0 \;. 
\end{equation}
By 
 \eqref{JtE1b}, a direct calculation gives
\[
\frac{\partial^2 J_c^*}{\partial \ell^2} \Big|_{(L(c),c)} = \frac{1}{1-e^{-\ell}} \frac{\partial^2 {\cal J}}{\partial \ell^2} \Big|_{(L(c),c)} \;.
\]
{As a consequence of Lemma~\ref{lem_pulse3}, this solution $(L(c),c)$ has to be a local minimizer of ${\cal J}$ and we use $c_p$ to designate the speed $c$ obtained from the above calculation. In the remainder of the proof, let $c=c_p$.
Then it follows from \eqref{H2b} and \eqref{Della} that 
\begin{equation} \label{sign1}
\frac{\partial {\cal J}}{\partial \ell}= e^{-\ell}\left(- \frac{\sqrt{2}}{12} (1+\alpha) + (1+H(c_p)) \frac{\sigma}{2 \gamma} (1 - e^{-r_2 \ell}) \right)
\left\{  \begin{array}{ll} >0, & \mbox{if} \; \ell > L(c_p), \\
                                  <0, & \mbox{if} \; \ell< L(c_p), \end{array} \right.
\end{equation}
from which we know 
$L(c_p)$ is a global minimizer of ${\cal J}(\cdot,c_p)$; thus denote $L(c_p)$ by $\ell_p$, the notion corresponding to $\ell_c$ when $c=c_p$.
 If we take $c=c_p$ and $b= - \log (1-e^{-\ell_p})$ then $J_{c}^*$ has a minimizer $\chi_{[b-\ell_p,b]}$ with 
 $J_{c}^*(\chi_{[b-\ell_p,b]})=0$ and $\| \chi_{[b-\ell_p,b]} \|_{L^1_e}=1$. 
 It follows from Lemma \ref{lem_width} that $\ell_p > \log  \frac{\alpha+1}{\alpha-1}$.}  \\

In summary, we have proved the following lemma under the assumption ($A 2$). 

\begin{lemma} \label{lem_pulse10}
Let 
$\sigma$ and $\gamma$ be given. If (A2) is satisfied,
 there exists a $c_p > 0$ such that, when $c=c_p$, $J_{c}^*$ has a 
 minimizer $\chi_{[a,b]}$ with $\| \chi_{[a,b]} \|_{L^1_e}=1$ and $J_{c}^*(\chi_{[a,b]})=0$. This $c_p$ is referred to as a $\Gamma$-limit 
  speed for the traveling pulse solutions and $b-a > \log  \frac{\alpha+1}{\alpha-1}$.
\end{lemma}

{We next study the uniqueness of $c_p$, as $\Gamma$-limit speed. This will in turn
implies the   
uniqueness of both $L(c_p)$  and  the minimizer of $J_c^*$. In the first step, we prove the following lemma.}

\begin{lemma} \label{lem_DcI}
Let $\ell,c>0$. 
Then  
$$
\frac{\partial{\cal J}}{\partial c}(\ell,c)<0\;.
$$
\end{lemma}
\begin{proof}
If $\ell,c>0$, it follows from  \eqref{JtE2}, the first equation in \eqref{H1aa} and \eqref{r21} that 
\begin{align*}
\frac{\gamma}{\sigma}\frac{\partial{\cal J}}{\partial c}(\ell,c)
&=\frac{\partial}{\partial c}\Big[\frac{1}{r_2-r_1}\big(r_2+r_1e^{-\ell}+e^{r_1\ell}\big)\Big] 
\\
&=\frac{\partial}{\partial c}\Big[\frac{1}{r_2-r_1}\big(r_1(1+e^{-\ell})+e^{r_1\ell}\big)\Big] 
\\
&=\frac{\partial}{\partial c}\Big[H(c)\big(r_1(1+e^{-\ell})+e^{r_1\ell}\big)\Big] 
\\
&=H'(c)\big(r_1(1+e^{-\ell})+e^{r_1\ell}\big)+H(c)\frac{\partial r_1}{\partial c}(1+e^{-\ell}+\ell e^{r_1\ell})
\\
&=\frac{4\gamma}{(c^2+4\gamma)^{3/2}}\big(r_1(1+e^{-\ell})+e^{r_1\ell}\big)+\frac{2\gamma}{c(c^2+4\gamma)}(1+e^{-\ell}+\ell e^{r_1\ell})
\\
&=\frac{4\gamma}{(c^2+4\gamma)^{3/2}}\Big[r_1(1+e^{-\ell})+e^{r_1\ell}+\frac{1+e^{-\ell}+\ell e^{r_1\ell}}{2H(c)}\Big] \;.
\end{align*}
It is easy to verify that $2 r_1 H(c)=-1-H(c)$. Using such an identity we obtain
\begin{eqnarray*}
\frac{\gamma}{\sigma}\frac{\partial{\cal J}}{\partial c} &=&
\frac{2\gamma}{ H(c) (c^2+4\gamma)^{3/2}}\Big[-(1+H(c))(1+e^{-\ell})+2 H(c)e^{r_1\ell}+1+e^{-\ell}+\ell e^{r_1\ell}\Big] \\
&=&
\frac{2\gamma}{ (c^2+4\gamma)^{3/2}}  {\cal K} \;,
\end{eqnarray*}
where
\begin{equation} \label{calK}
{\cal K}(\ell,c) \equiv -1-e^{-\ell}+2 e^{r_1\ell}+\frac{\ell}{H(c)} e^{r_1\ell} \;.
\end{equation}
A direct calculation yields
\begin{eqnarray}
\frac{\partial {\cal K}}{\partial \ell} & =  &e^{-\ell}+ 2 r_1 e^{r_1 \ell} +\frac{1}{H(c)} ( e^{r_1 \ell}+ r_1 \ell e^{r_1 \ell} ) \nonumber  \\
&=& e^{-\ell} - e^{r_1 \ell}+ \frac{r_1 \ell}{H(c)} e^{r_1 \ell} \;. \label{K_ell}
\end{eqnarray}

It suffices to show ${\cal K}<0$ for all $\ell,c>0$. 
Note that ${\cal K}(0,c)=0$. Then for small $\ell$,
 a Taylor's expansion gives
\begin{eqnarray*}
{\cal K} &= & -1 - (1- \ell + \frac{\ell^2}{2}) 
+ 2 (1+ r_1 \ell + \frac{r_1^2 \ell^2}{2}) +\frac{\ell}{H(c)} (1+r_1 \ell) + O(\ell^3) \\
&=& \ell^2(-\frac{1}{2}+r_1^2 +\frac{r_1}{H(c)}) + O(\ell^3) \\
&=& \ell^2 ( -\frac{1}{2} + r_1 r_2) + O(\ell^3) \\
&=& \ell^2 ( -\frac{1}{2} -\frac{\gamma}{c^2}) + O(\ell^3) \\
& <& 0 \;.
\end{eqnarray*}
Therefore ${\cal K}=\frac{\partial {\cal K}}{\partial \ell}=0$ and $\frac{\partial^2 {\cal K}}{\partial \ell^2} <0$ at $\ell=0$, which implies
$\frac{\partial {\cal K}}{\partial \ell}<0$ for small positive $\ell$. Besides $\ell=0$, we claim that  there exists another unique non-negative root
of $\frac{\partial {\cal K}}{\partial \ell}$. Indeed such a root satisfies
\[
e^{r_2 \ell}= 1 - \frac{r_1}{H(c)} \ell \;,
\]
which correponds to an intersection point of
the exponential function 
$y=e^{r_2 \ell}$ and the straight line $y=1 - \frac{r_1}{H(c)} \ell$ in the $(\ell, y)$ plane. \\

As $\frac{\partial {\cal K}}{\partial \ell}<0$ for small positive $\ell$, the straight line lies above the graph of exponential function near $\ell=0$.
On the other hand
the exponential function 
 will dominate the straight line when $\ell$ is large, hence there exists another unique intersection point,
which is called $\ell_1$. Clearly $\frac{\partial {\cal K}}{\partial \ell}<0$ for $\ell< \ell_1$ and $\frac{\partial {\cal K}}{\partial \ell}>0$ for 
$\ell > \ell_1$. As $\ell \to \infty$, it is easily checked that ${\cal K} \to -1$. We can now conclude that ${\cal K}$ dips  below $0$ near 
$\ell=0$, reaching a negative minimum at $\ell=\ell_1$, then increases again to reach ${\cal K}=-1$ at $\ell=\infty$. Hence for any given $c>0$,
we know ${\cal K}<0$ for $\ell>0$.
This concludes the proof of the lemma.
\end{proof}

\begin{lemma} \label{lem_unique_c}
Under the assumption of Lemma~\ref{lem_pulse10}, the $\Gamma$-limit speed $c_p$ is unique.
\end{lemma}

\begin{proof}
We argue indirectly. From Lemma~\ref{lem_pulse10}, suppose that 
$\chi_{[a_i,b_i]}$ is a minimizer of $J_{c_i}^*$ for $i=1,2$ and 
$\ell_i=b_i-a_i$ with
$\ell_2>\ell_1$. By \eqref{L'}, it is necessary that  $c_2>c_1$. Then 
${\cal J}(\ell,c_i)>{\cal J}(\ell_i,c_i)=0$ for all $\ell \ne \ell_i$.
Together with $\frac{\partial {\cal J}}{\partial c}<0$, 
we obtain $0={\cal J}(\ell_1,c_1)>{\cal J}(\ell_1,c_2)>{\cal J}(\ell_2,c_2)=0$, which is absurd.
\end{proof}
\begin{remark} \label{remark_cp}
{As a consequence of Lemma~\ref{lem_pulse10} and Lemma~\ref{lem_unique_c}, when consider a positive speed for the $\Gamma$-limit, 
 $c_p$ is the unique speed and $J_{c_p}$ has a unique minimizer
$\chi_{[a,b]}$.} 
\end{remark}

{The proof of Theorem~\ref{thm_epsilon_pulse} is analogous to that of 
Theorem~\ref{thm_front}, we omit it. 
Finally recall $\gamma_*$ from 
Lemma~\ref{lem_gstar} and note that when $\epsilon$ is sufficiently small, 
 ($A 2$) is equivalent to}
\begin{equation} \label{phy_pulse}
\alpha>1 \quad \mbox{and} \quad  \gamma_* >\gamma. 
\end{equation}


\vspace{.3in}
\noindent
{\bf \large Acknowledgments}
Research is supported in part  by the Ministry of Science and Technology. Part of the work was done when 
Choi and Fusco were visiting the National Center for Theoretical Sciences, Taiwan.


\bibliographystyle{plain}

\end{document}